\documentclass{article} % use \documentstyle for old LaTeX compilers
\usepackage{arxiv}
\usepackage[english]{babel} % 'french', 'german', 'spanish', 'danish', etc.
\usepackage{amssymb}
\usepackage{amsmath}
\usepackage{amsthm}
\usepackage{mathtools}
\usepackage{cleveref} 
\usepackage{abstract}

\usepackage{xcolor}
\usepackage{breqn,url}
\usepackage{tikz}
\usepackage{pgfplots}
\usepackage{graphicx}
\pgfplotsset{width=8cm,compat=1.9}
\usepackage{caption}
\usepackage{subcaption}

\newtheorem{theorem}{Theorem}
\newtheorem{lemma}{Lemma}

\newtheorem{proposition}{Proposition}
\newtheorem{remark}{Remark}
\newtheorem{example}{Example}

\title{Construction of multivariate polynomial approximation kernels via semidefinite programming}

\author{
	Felix Kirschner
	\thanks{Tilburg University,	f.c.kirschner@tilburguniversity.edu   } 
	\and 
	\textbf{Etienne de Klerk}
	\thanks{Tilburg University,	e.deklerk@tilburguniversity.edu
	\newline This work is supported by the European Union’s Framework Programme for Research and Innovation Horizon 2020 under the Marie Sk\l odowska-Curie grant agreement N. 813211 (POEMA).}
}

\begin{document}
	\maketitle
	
	\begin{abstract}
		In this paper we construct a hierarchy of multivariate polynomial approximation kernels for uniformly continuous functions on the hypercube via semidefinite programming. 
  We give details on the implementation of the semidefinite programs defining the kernels. 
  Finally, we show how symmetry reduction may be performed to increase numerical tractability. 
	\end{abstract}
	
	\section{Introduction}
	A classical problem in approximation theory  is uniform  approximation of a given function by linear combinations of orthogonal polynomials. In the following we will denote by $\mathbf{K}$ the $n$-dimensional hypercube, i.e. $\mathbf{K} := [-1,1]^n$. Orthogonality of polynomials may be defined in the following way. Let $\mu$ be a positive finite Borel measure supported on the compact set $\mathbf{K} = [-1,1]^n \subset \mathbb{R}^n$.
	 We say two functions $f,g \in \mathcal{C}(\mathbf{K})$ are orthogonal (w.r.t.\ $\mu$), if
	\[
	\langle f, g \rangle_{\mu} := \int_{\mathbf{K} }f(\textbf{x})g(\textbf{x})\mathrm{d}\mu(\textbf{x}) = 0.
	\]
	Let $\{ p_{\alpha}\}_{\alpha \in \mathbb{N}^n}$ be a system of orthogonal polynomials with respect to a measure $\mu$.
	 Consider a kernel $K_r(\textbf{x},\textbf{y}) : \mathbb{R}^n \times \mathbb{R}^n \rightarrow \mathbb{R}$ given by
	\begin{equation}\label{polKernel}
	K_r(\textbf{x},\textbf{y}) := \sum_{\alpha \in \mathbb{N}^n_r} g_\alpha p_\alpha(\textbf{x})p_\alpha(\textbf{y}),
	\end{equation}
	for given constants $g_\alpha$ for $\alpha \in \mathbb{N}^n_r$,
	where  $\mathbb{N}^n_r = \{ \alpha \in \mathbb{N}^n : \alpha_1 + \dots + \alpha_n \le r \}$.
	Then the convolution operator, defined as
	\begin{equation}\label{convOperator}
	\mathcal{K}^{(r)}(f)(\textbf{x}) := \int_{\mathbf{K}} f(\textbf{y})K_r(\textbf{x},\textbf{y})\mathrm{d}\mu(\textbf{y}),
	\end{equation}
	maps any $\mu$-integrable function $f$ to a polynomial of degree at most $r$. More precisely,
	\begin{equation}\label{calKoperator}
	\mathcal{K}^{(r)}(f)(\textbf{x})  = \sum_{\alpha \in \mathbb{N}^n_r} b_\alpha p_\alpha(\textbf{x}) \,,
	\mbox{ where }
	b_\alpha = \langle p_\alpha, f \rangle_\mu\, g_\alpha.
	\end{equation}
	The coefficients $g_\alpha$ of the kernel $K_r$ determine the approximation.  
	Our goal is to approximate a given continuous $f$ defined on $\mathbf{K}$, by a sequence of polynomials of increasing degree, such that the sequence converges to $f$, uniformly on $\mathbf{K}$.
	We further introduce a quantity
	$$\sigma_r := \left(\int_{\mathbf{K} \times \mathbf{K}} \|\textbf{x}-\textbf{y}\|^2 K_r(\textbf{x},\textbf{y}) \mathrm{d}\mu(\textbf{x})\mathrm{d}\mu(\textbf{y})\right)^{1/2},$$  called the {\em resolution} of the kernel $K_r(\textbf{x},\textbf{y})$ in \cite{Weisse2006}. Our aim in this paper is to construct kernels $K_r$ with minimal resolution such that $\mathcal{K}^{(r)}(f)$ converges to $f$ uniformly on $\mathbf{K}=[-1,1]^n$ and to bound the rate of convergence in terms of $\sigma_r$. The resolution may be interpreted as a measure of how much \emph{mass} of the kernel is concentrated away from the line where $\textbf{x} = \textbf{y}$. If all $g_\alpha = 1$ in expression \eqref{calKoperator} then $\mathcal{K}^{(r)}$ is the identity operator on the space of polynomials of degree at most $r$, and the associated resolution is zero.. This kernel will have all of its mass concentrated at $\textbf{x} = \textbf{y}$. To ensure uniform convergence we want a kernel that has as much mass as possible at the line $\textbf{x}= \textbf{y}$ for every $r \in \mathbb{N}$, while fulfilling some other properties. %, and the rate of convergence may be bounded in terms of $\sigma_r$. %We give a proof later for completeness.

	\subsection*{ Outline and contributions}
	The aim of this paper is to present a computational procedure, based on semidefinite programming (SDP) (cf. \cite{Todd01,Vandenberghe1996}), to construct non-negative polynomial kernels on $\mathbf{K} = [-1,1]^n$ suitable for approximation. 
	%These kernels satisfy the properties P1 through P4 above, where the resolution $\sigma_r$  is minimal (in a well-defined sense). 
	We show that these kernels generalize a kernel which is called the Jackson kernel in \cite{Weisse2006}, but this is different from the original kernels introduced by Jackson in \cite{Jackson1912}; we give more details on this in \cref{sec:construction optimal kernels}.
	As the orthogonal basis we will use products of univariate Chebyshev polynomials, as reviewed in \cref{subsec:Chebyshev}, and the fixed measure $\mu$ will be the corresponding product of measures so that the Chebyshev polynomials are orthogonal. The resulting {\em kernel polynomial method} is reviewed in \cref{sec:KPM} for the univariate case $(n=1)$, and extended to the multivariate case in \cref{sec:multivariate KPM}. In \cref{sec:implementation} we discuss how to form the SDP problems that yield the optimal kernels, in the sense that their resolution is minimal. In \cref{sec:symmetry} we show how to exploit algebraic symmetry by using techniques from \cite{Riener2013,Vallentin2009} to reduce the size of these SDP problems.
	We show in \cref{sec:comparison to univariate Jackson} that our constructions are superior to simply multiplying optimal univariate kernels in a well-defined sense. 
	Finally, in \cref{sec:numComp} we give further details of our numerical computations and show they are useful in practice to approximate non-differentiable functions and related applications in physics.

	\subsection{Prior and related work} 
	Setting all coefficients of the kernel $K_r$ equal to $g_\alpha = (\langle p_\alpha, p_\alpha \rangle_\mu)^{-1}$, the resulting approximation in the setting described above is simply the Chebyshev expansion truncated at degree $r$. In the univariate case, the resulting kernel is known as the Dirichlet kernel \cite{Dirichlet1829}. In the multivariate case it is known as the Christoffel-Darboux kernel (named after \cite{Christoffel1858, Darboux1878}). This approximation works well for analytic functions, as reviewed in \cite{Trefethen2017}. Lasserre \cite{Lasserre2021} draws an interesting connection between the celebrated moment-SOS hierarchy \cite{Lasserre2001} and the Christoffel-Darboux kernel. For non-differentiable functions, the truncated Chebyshev expansion, i.e., the Christoffel-Darboux kernel, may lead to unwanted oscillations at points where the function is not differentiable, as reviewed in \cite{Weisse2006}. These oscillations are often referred to as the Gibbs phenomenon, see \cite{Gottlieb1997} for a survey. In \cite{Marx2021} the authors develop a method for approximating possibly discontinuous functions using the Christoffel-Darboux kernel, where the Gibbs phenomenon does not occur. Other approaches to get rid of unwanted oscillations is to make use of non-negative kernels as has been done in \cite{Weisse2006}. For this reason, positive approximation kernels are popular in physics for the approximation of non-smooth functions in various settings. The reader is again referred to the excellent survey \cite{Weisse2006} for more details. The aim of our work is to generalize these kernels to several variables in a natural way, thus providing computational alternatives to using products of univariate kernels.

	%\section{Preliminaries and notation}
	We continue to fix some notation, and the review properties of univariate Chebyshev polynomials for later use.
	Our exposition closely follows the survey \cite{Weisse2006}.

	\subsection{Notation}
	 For $1 < n \in \mathbb{N}$ we define $[n] :=\{1, 2, \dots, n \}$. The canonical unit vectors in $\mathbb{R}^n$ are denoted by $e_i$, i.e., $e_i$ is the zero vector with a $1$ at the $i$-th component. The polynomial ring is denoted by $\mathbb{R}[\textbf{x}]$, where $\textbf{x} = (x_1, \dots, x_n)$. A polynomial $p \in \mathbb{R}[\textbf{x}]$ is said to be a \emph{sum-of-squares} if it can be written as a sum of squared polynomials, i.e., $p(\textbf{x}) = \sum_{i=1}^k q_i(\textbf{x})^2$ for $k \in \mathbb{N}$ and $q_i \in \mathbb{R}[\textbf{x}]$ for all $i \in [k]$. The set of polynomials of degree less than or equal to $r$ will be denoted by the set $\mathbb{R}[\textbf{x}]_r$. For a set $\mathbf{K} \subseteq \mathbb{R}^n$ we denote by $\mathcal{C}(\mathbf{K})$ the set of continuous functions on $\mathbf{K}$. $\mathbb{S}^n$ is the set of $n \times n$ symmetric matrices, and $\mathbb{S}^n_{\succeq}$ is the set of symmetric positive semidefinite (psd) matrices. A symmetric matrix $S$ is called positive semidefinite if $x^TSx \ge 0$ for all $x \in \mathbb{R}^n$. We may write $S \succeq 0$ to indicate $S$ is psd. For $\alpha, \beta \in \mathbb{N}^n$ the Kronecker-$\delta_{\alpha, \beta}$ is defined as 
	\[
	\delta_{\alpha,\beta} = \begin{cases} 1, \text{ if } \alpha_i = \beta_i \; \forall i \in [n] \\ 0, \text{ otherwise.} \end{cases}
	\]
	For two matrices $A,B$ of appropriate size we define the trace inner product $\langle A, B \rangle := \mathrm{Tr}(A^TB)$. We also set 
	\[
	s(n,r) = \binom{n+r}{r}.
	\]

	\subsection{Chebyshev polynomials}\label{subsec:Chebyshev}
	
	Let $\mathbf{K} = [-1,1]$  and fix the measure $\mu$ on $\mathbf{K}$ defined by $\mathrm{d}\mu(x) = (\pi \sqrt{1-x^2})^{-1}\mathrm{d}x$, $x \in \mathbf{K}$. The Chebyshev polynomials of the first kind form a system of orthogonal polynomials. We will refer to the $k$-th Chebyshev polynomial of first kind as $T_k(x)$. We have for $k \in \mathbb{Z}$
	\begin{align}
	T_{k}(x) & = \cos \left( k \arccos \left(x\right) \right) .
	%U_k(x) &= \frac{\sin((k+1)\arccos(x))}{\sin(\arccos(x))}
	\end{align}
	
	Define for $f,g : [-1,1] \rightarrow \mathbb{R}$
	\[ \langle f,g \rangle_{\mu} = \int_{-1}^1 \frac{f(x)g(x)}{\pi \sqrt{1-x^2} } \mathrm{d}x \]
	% and
	% \[ \langle f,g \rangle_2 = \int_{-1}^1 f(x)g(x)\left( \pi \sqrt{1-x^2} \right) \mathrm{d}x \]
	to obtain the following orthogonality relations for the Chebyshev polynomials of the first kind
	\begin{align}
	\langle T_k, T_m \rangle_{\mu} &= \frac{1+\delta_{k,0}}{2} \delta_{k,m}.% \\
	%\langle U_k, U_m \rangle_2 &= \frac{\pi^2}{2} \delta_{k,m}.
	\end{align}
	Chebyshev polynomials exhibit nice stability and convergence properties in practice which is why they are the first choice in many applications. It is straightforward to generalize the Chebyshev polynomials to the multivariate case. Let $\mathbf{K} = [-1,1]^n$ and define
	\[
	\mathrm{d}\mu(\textbf{x}) := \prod_{i=1}^{n} \frac{1}{\pi \sqrt{1-x_i^2}}\mathrm{d}\textbf{x}.
	\]
	% and
	% \[
	% \mathrm{d}\mu_2(\textbf{x}) := \prod_{i=1}^{n} \pi \sqrt{1-x_i^2}\mathrm{d}\textbf{x}.
	% \]
	Then, for $\alpha \in \mathbb{N}^n$ the corresponding multivariate Chebyshev polynomial of the first kind is defined as
	\[
	T_\alpha(\textbf{x}) = \prod_{i=1}^n T_{\alpha_i}(x_i).% \quad U_\alpha(\textbf{x}) = \prod_{i=1}^n U_{\alpha_i}(x_i),
	\]
	
	The orthogonality relations extend in the following way
	
	\begin{equation*}
		\begin{aligned}
		\langle T_\alpha, T_\beta \rangle_{\mu} = \int_{\mathbf{K}}T_\alpha(\textbf{x})T_\beta(\textbf{x}) \mathrm{d}\mu(\textbf{x})
		 & = \prod_{i=1}^{n} \int_{-1}^{1} \frac{T_{\alpha_i}(x_i)T_{\beta_i}(x_i)}{\pi \sqrt{1-x_i^2}}\mathrm{d}x_i \\
		& = \prod_{i= 1}^{n} \frac{1+\delta_{\alpha_i,0}}{2}\delta_{\alpha_i,\beta_i} = c_\alpha \delta_{\alpha, \beta},
		\end{aligned}
	\end{equation*}

	with $c_\alpha = \left(\frac{1}{2}\right)^{H(\alpha)}$, 
	where $H(\alpha)$ is the Hamming weight of $\alpha$, i.e. the number of non-zero entries.
	%The orthogonality relations for the Chebyshev polynomials of the second kind may be derived in a similar way.

	\subsection{Main result}

	Using the Chebyshev polynomials we are ready to state our main result. We will be interested in kernels $K_r(\textbf{x},\textbf{y})$ satisfying the following properties for $\mathbf{K} = [-1,1]^n$:

	%S. Bernstein 2, Sur l'ordre de la meilleure approximation des fonctions continues par des polynômes de degré donné, Mémoire couronné, Brussels, 1912.
	
	\begin{enumerate}
		\item[P1.]
		$K_r(\textbf{x},\textbf{y}) = \sum_{\alpha \in \mathbb{N}^n_r}g_\alpha T_\alpha(\textbf{x})T_\alpha(\textbf{y})$, for $g_\alpha \in \mathbb{R}$ for $ \alpha \in \mathbb{N}^n_r$
		\item[P2.]
		$K_r(\textbf{x},\textbf{y}) \ge 0$ for all $(\textbf{x},\textbf{y}) \in \mathbf{K} \times \mathbf{K}$ and all $r$;
		\item[P3.]
		$  \int_{\mathbf{K} } K_r(\textbf{x},\textbf{y}) \mathrm{d}\mu(\textbf{y}) = 1$ for all  $\textbf{x} \in \mathbf{K}$ for all $r$;
		\item[P4.]
		$\lim_{r \rightarrow \infty} \sigma_r = 0$.
	\end{enumerate}
	% , and for the second one we define
	% $$\bar \sigma_r :=  \sup_{\textbf{x} \in \mathbf{K}} \left(\int_{\mathbf{K} } \|\textbf{x}-\textbf{y}\|^2 K_r(\textbf{x},\textbf{y}) \mathrm{d}\mu(\textbf{y})\right)^{1/2}.$$
	
	% Note that it readily follows from P4 that 
	% \[
	% 0 \le    \sigma_r  \rightarrow 0 \mbox{ as } r \rightarrow \infty.
	% \]

	%Suppose we are given a kernel $K_r$ satisfying properties P1, P2 and P3. 
	In the statement of \cref{prop1}, recall that the modulus of continuity of $f \in C(\mathbf{K})$ is defined as
	\[
	\omega_f(\delta) := \max_{\stackrel{\textbf{x},\textbf{y} \in \textbf{K}}{\|\textbf{x}-\textbf{y}\| \le \delta}} |f(\textbf{x}) - f(\textbf{y})|.
	\]

	\begin{proposition}\label{prop1}
		Let $\mathbf{K}=[-1,1]^n$ and  $f:\mathbf{K} \rightarrow \mathbb{R}$ be continuous on $\mathbf{K}$ with modulus of continuity $\omega_f$.
		Under the above  conditions P1-P4 on $K_r(\textbf{x},\textbf{y})$ one has $\mathcal{K}^{(r)}(f) \rightarrow f$ as $r \rightarrow \infty$, uniformly on $\mathbf{K}$.
		Moreover,   
		\begin{equation}
		 \label{ineq:bar resolution}
		\|\mathcal{K}^{(r)}(f) - f\|_{\infty,\mathbf{K}} \le  2\left(1+\frac{\pi}{\sqrt{2}}\right)\omega_f({\sigma_r}).
		\end{equation}
		% Furthermore, one has
		% \begin{equation}
		% 	\label{ineq:resolution}
		% \int_{\mathbf{K}} | \mathcal{K}^{(r)}(f)(\textbf{x})- f(\textbf{x})|d\mu(\textbf{x}) \equiv 	\|\mathcal{K}^{(r)}(f) - f\|_{1,\mathbf{K}} \le 2\omega_f({ \sigma_r}),
		% \end{equation}
		% i.e.\ the kernel resolution determines the rate of convergence in the $\ell_1$-norm.
	\end{proposition}

  Our main result is the construction of kernels whose resolutions satisfy $\sigma_r = O(1/r)$ using semidefinite programming techniques (see \cref{propConvRate}). This proves that our kernels yield the best possible rate of convergence for continuous $f$ that are not differentiable, due to Bernstein's theorem (see \cite{Bernstein1912}). The proof of \cref{prop1} will be postponed to later, until we have given all necessary definitions and auxiliary results. Let us mention at this point that \cref{prop1} is a known result in approximation theory. Indeed, the argument is essentially as given in the PhD thesis of Jackson \cite{Jackson1911}. We simply give a proof in our specific setting for completeness, since we could not find a statement of \cref{prop1} in a suitable form in the literature.

	\section{The kernel polynomial method}
	\label{sec:KPM}
	We begin by considering kernels to approximate univariate functions. Let $K_r$ be a kernel of the following form
	\begin{equation}\label{Kr}
	K_r(x,y) = g_0+2\sum_{k=1}^{r} g_k T_k(x)T_k(y).
	\end{equation}
	 Kernels of this kind clearly satisfy property P1. If we set $g_0 = 1$, the resulting kernel also satisfies P3. In the following, we will explore how to find kernels that of this form that additionally satisfy P2 and are therefore suitable for approximation. For this we will first introduce trigonometric polynomials.
	
	\subsection{Trigonometric polynomials}
	
	A trigonometric polynomial  $p(t)$ of degree $r$ is defined as
	\[ p(t) = p_0 + \sum_{k=1}^r \left(p_k \cos \left( kt  \right) + p_{-k}\sin \left(kt \right) \right), \]
	for $p_k \in \mathbb{R}$ for $k = -r, -r+1, \dots, r-1,r$. The following lemma is proved in \cite{Gatermann2004}.
	
	\begin{lemma}\label{nonnegtrigpol}
		If $p(t)$ is a non-negative trigonometric polynomial of degree $r$, then there exists a positive semidefinite matrix $Q\in \mathbb{S}^{r+1}_{\succeq}$ such that $p(t) = v^TQv$ where
		\[ v^T = [1, \cos (t), \dots, \cos (mt), \sin (mt) ] \]
		if $r = 2m$ for some $m \in \mathbb{N}$
		and
		\[ v^T = \left[\cos \left(\frac{t}{2}\right), \sin \left(\frac{t}{2}\right), \cos \left(t+\frac{t}{2}\right) , \dots, \cos \left(mt+\frac{t}{2}\right) , \sin \left(mt+\frac{t}{2}\right) \right] \]
		if $r = 2m+1$ for some $m \in \mathbb{N}$.
	\end{lemma}

	\begin{remark}
		\sloppy Let us mention that there are stronger results of the kind of \cref{nonnegtrigpol}. 
		For example Corollary 2 in \cite{Foucart2017}. We state this weaker result for the ease of exposition.
	\end{remark}
	
	Note that every trigonometric polynomial of the form
	
	\begin{equation}\label{trigpol} p(t) = g_0 + 2\sum_{k=1}^{r} g_k \cos \left( k t \right)  \end{equation}
	
	gives rise to a kernel of the form
	
	\[ K_r(x,y) = g_0 + 2 \sum_{k=1}^r g_k T_k(x)T_k(y). \]
	
	To see this, consider the following substitution

% HERE HERE	
	\begin{multline}\label{subst}
      \frac{1}{2} \left[ p(\arccos (x) + \arccos (y)) + p(\arccos (x) - \arccos (y)) \right]\\ 
      =  g_0 + 2 \sum_{k=1}^{r}g_k \frac{1}{2}\left[ \cos \left( k ( \arccos (x) + \arccos (y))\right) + \cos  \left( k ( \arccos (x) - \arccos (y))\right)  \right] \\
      = g_0 + 2 \sum_{k=1}^{r} g_k \cos \left( k \arccos (x) \right)  \cos \left(k \arccos (y) \right) = g_0 + 2 \sum_{k=1}^{r} g_k T_k(x) T_k(y). \qquad
\end{multline}
	
	If $p(t)$ is non-negative on $[-\pi,\pi]$, then $K_r(x,y)$ is non-negative on $[-1,1]^2$.

	\begin{theorem}\label{fejer}{(Fej\'er (1915))}
		Every non-negative trigonometric polynomial of degree $r$ of the form \[ p(t) = \lambda_0 + \lambda_1 \cos t  + \mu_1 \sin t  + \dots + \lambda_r \cos rt  + \mu_r \sin rt  \]
		can be written as \[ p(t) =  \left \vert \sum_{\nu = 0}^{r} c_\nu \mathrm{e}^{i\nu t} \right \vert^2 	\] for $c_\nu \in \mathbb{C}$.
	\end{theorem}
	
	In other words, there is a one-to-one correspondence between trigonometric polynomials of the form
	\[ t \mapsto \lambda_0 + \lambda_1 \cos (t)  + \mu_1 \sin (t)  + \dots + \lambda_r \cos (rt)  + \mu_r \sin (rt)  \]
	that are non-negative for every $t$ and functions of the form
	\[t \mapsto \left \vert \sum_{\nu = 0}^{r} c_\nu \mathrm{e}^{i\nu t} \right \vert^2.  \]
	This correspondence may be leveraged to obtain kernels with minimum resolution, which is done in the next section.
	
	\subsection{Constructing optimal kernels}\label{sec:construction optimal kernels}
	In this subsection we will revisit the approach described in \cite{Weisse2006}, showing the kernel they obtain has minimum resolution among all non-negative kernels on $[-1,1]^2$. 
	To avoid ambiguity, note the following. The authors in \cite{Weisse2006} refer to their kernel as the \emph{Jackson kernel}, even though in the literature there is another object which is referred to in that name. 
	Therefore, we will refer to the kernel from \cite{Weisse2006} as the \emph{minimum resolution kernel}, reserving the term "Jackson kernel" for the object Jackson used in \cite{Jackson1912} to prove his theorems. 
	We are interested in non-negative trigonometric polynomials with cosine terms only, as these are the ones giving rise to kernels of the form that we want as we have seen in \eqref{trigpol}, \eqref{subst}. It is easy to see that if all sine-terms are zero, then the $c_\nu$ terms are real. Thus, this gives us a way to characterize kernels of the form \eqref{Kr} that are non-negative.  Consider a function of the form
	\[
	p(\varphi) = \left \vert \sum_{\nu=0}^r a_\nu \mathrm{e}^{i\nu \varphi} \right \vert^2,
	\]
	for $a_\nu \in \mathbb{R}$. Rewriting this expression we find
	\begin{equation*}
		\begin{aligned}
		p(\varphi) &= \sum_{\nu, \mu = 0}^{r} a_\nu a_\mu \cos\left( [\mu - \nu]\varphi \right)\\
						&=  \sum_{\nu=0}^{r}a_\nu^2 + \sum_{k=1}^{r} \sum_{\nu = 0}^{r-k}a_\nu a_{\nu+k}\cos (k \varphi) \\
						& = g_0 + \sum_{k = 1}^{r} g_k \cos (k \varphi),
		\end{aligned}
	\end{equation*}
	for
	\begin{equation}\label{gn}
	g_k = \sum_{\nu=0}^{r-k}a_\nu a_{\nu+k}.
	\end{equation}
	Therefore, every set of real numbers $a_0, a_1, \dots, a_r$ satisfying $\sum_{k = 0}^{r}a_k^2 = 1$ gives rise to a kernel of the form \eqref{Kr} satisfying P1, P2, P3 when the $g_k$ are set as in \eqref{gn}. 
	A first idea to construct a kernel would be to set all $a_k = \frac{1}{\sqrt{r+1}}$, to ensure that $g_0 = 1$. 
	This leads to the so called Fej\'er kernel, whose coefficients we denote by 
	\[
		g_k^F = 1-\frac{k}{r+1}.
	\]
	%Then we find $g_k^F = 1-\frac{k}{r+1}$. The resulting object is known as the Fej\'er kernel \cite{Fejer1904}. 
	However, this kernel is not optimal in the sense that is does not have minimum resolution. We next take a look at kernels satisfying P1-P3, with minimal resolution $\sigma_r$. Note that for the resolution of the kernel we have
	\[
	\sigma_r^2 = \int_{\mathbf{K}} (x-y)^2K_r(x,y)\mathrm{d}\mu(x)\mathrm{d}\mu(y) = g_0 - g_1.
	\]
	We will formulate an optimization problem to minimize resolution with respect to $a_k$.
	\begin{equation}\label{jacksonproblem}
	\begin{aligned}
		\min &\; g_0 - g_1  & \Leftrightarrow & \quad \min \sum_{k = 0}^r a_k^2 - \sum_{k = 0}^{r-1}a_ka_{k+1}\\
		\text{ s.t. } & g_0 = 1  & &\quad \text{ s.t. }  \sum_{k = 0}^r a_k^2  = 1
	\end{aligned}
	\end{equation}
	Solving this problem results in the minimum resolution kernel mentioned earlier which is called the Jackson kernel in \cite{Weisse2006}, whose coefficients are given by
	\begin{equation}\label{jackson}
		g_{k,r}^{\mathrm{KPM}} = \frac{(r-k+2)\cos \left( \frac{\pi k}{r+2}\right) +\sin \left( \frac{\pi k}{r+2}\right) \cot \left( \frac{\pi}{r+2}\right) }{r+2}.
	\end{equation}
	Thus, we see that $\sigma_r^2 = g_{0,r}^{\mathrm{KPM}} - g_{1,r}^{\mathrm{KPM}} = 1-\cos \left( \frac{\pi}{r+2} \right)  = O(1/r^2)$.

  \subsection{Uniform convergence in terms of $\sigma_r$ in the multivariate case}

	In this subsection we prove we may also bound the rate of convergence in terms of $\sigma_r$ in the multivariate case. 
	We first prove the result for uniformly continuous periodic functions on $[-\pi,\pi]^n$ and then extend the results for the case of uniformly continuous functions on $[-1,1]^n$.
	Recall that the modulus of continuity for a uniformly continuous function $f$ is defined as
	\[
	\omega_f(\delta) := \max_{\stackrel{\textbf{x},\textbf{y} \in \textbf{K}}{\|\textbf{x}-\textbf{y}\| \le \delta}} |f(\textbf{x}) - f(\textbf{y})|.
	\]
	Further, note the following properties of the modulus of continuity
	\begin{enumerate}
		%\item For $\delta_1, \delta_2 >0 : \omega_f(\delta_1+\delta_2) \le \omega_f(\delta_1)+\omega_f(\delta_2)$ (direct verfication)
		\item\label{itemModCont1} For $\lambda, \delta > 0 : \omega_f(\lambda \delta) \le (1+\lambda)\omega_f(\delta)$ (Lemma 1.3 in \cite{Rivlin69})
		\item\label{itemModCont2} For continuous $f$ and $\delta > 0$ one has $|f(\textbf{x}-\textbf{y})-f(\textbf{x})| \le \left( 1+\frac{1}{\delta^2}||\textbf{y}||^2\right)\omega_{f}(\delta)$  (by proof of Proposition 5.1.5 in \cite{Altomare2011})
	\end{enumerate}
    Let $f$ be uniformly continuous and periodic on $\textbf{B} := [-\pi, \pi]^n$ and let the kernel 
    \begin{equation}\label{trigKernel}
    K_r(\textbf{x}) = 1+\sum_{\alpha \in \mathbb{N}^n_r \setminus \{ 0 \}} 2^{H(\alpha)}g_\alpha \prod_{i = 1}^n \cos \left( \alpha_i x_i \right)
    \end{equation}
    
    be non-negative of degree $r$. Define
    \[
        \mathcal{K}^{(r)}(f)(\textbf{x}) = \frac{1}{(2\pi)^n}\int_{\textbf{B}}f(\textbf{x}-\textbf{y})K_r(\textbf{y}) \mathrm{d}\textbf{y}.
    \]
    
    Further, note that
	\begin{equation}\label{inequality}
    \begin{aligned}
        | \mathcal{K}^{(r)}(f)(\textbf{x})- f(\textbf{x})| &\le \frac{1}{(2\pi)^n} \int_{\textbf{B}} | f(\textbf{x}-\textbf{y})-f(\textbf{x})|K_r(\textbf{y})\mathrm{d}\textbf{y} \\ 
        &\overset{\ast}{\le}  \frac{1}{(2\pi)^n} \int_{\mathbf{B} }  \left(1+\frac{1}{\delta^2}\|\textbf{y}\|^2\right)\omega_f(\delta)  K_r(\textbf{y}) \mathrm{d}\textbf{y} \\
        &= \omega_f(\delta) + \frac{\omega_f(\delta)}{\delta^2} \frac{1}{(2\pi)^n}\int_{\mathbf{B}} \|\textbf{y}\|^2 K_r(\textbf{y}) \mathrm{d}\textbf{y},
    \end{aligned}
\end{equation}
    where $\ast$ follows from the second property of the modulus of continuity. Note that since $-\pi \le y_i \le \pi$ we find
    \begin{align*}
        ||\textbf{y}||^2 &= \sum_{i=1}^n y_i^2 \le  \frac{\pi^2}{2}\sum_{i=1}^n (1-\cos (y_i) ) \,. \\
        %& = \sum_{i=1}^n 1-\cos(y_i)
       % & \le 2\sum_{i=1}^n 1- \sin(x_i)-\cos(x_i) 
    \end{align*}

    Thus, 
    \begin{align}
        \int_{\mathbf{B}} \|\textbf{y}\|^2 K_r(\textbf{y}) \mathrm{d}\textbf{y} & \le \frac{\pi^2}{2}\left(\sum_{i=1}^n (2\pi)^n - \sum_{\alpha \in \mathbb{N}^n_r} 2^{H(\alpha)}g_\alpha \int_{\textbf{B}}\cos (y_i) \prod_{i = 1}^n \cos (\alpha_i y_i) \mathrm{d}\textbf{y}\right) \\
        & = (2\pi)^n \frac{\pi^2}{2}\sum_{i=1}^n \left(1-g_{e_i}\right) \,, 
    \end{align}
    where we used the fact that 

    $$\cos (x) \cos (kx)  = \frac{1}{2}\left(\cos ((k-1)x) +\cos ((k+1)x)  \right).$$ 

	Straight-forward calculation shows that $\sum_{i=1}^n (1-g_{e_i}) = \sigma_r^2$.	
    Choosing $\delta = \sqrt{ \frac{\pi^2}{2} \sum_{i=1}^n (1- g_{e_i}}) = \pi / \sqrt{2} \sigma_r$ we find by \eqref{inequality} and the first property of the modulus of continuity:
    \[
        | \mathcal{K}^{(r)}(f)(\textbf{x})- f(\textbf{x})| \le 2\omega_f(\pi / \sqrt{2} \sigma_r) \le 2(1+\pi/\sqrt{2})\omega_f(\sigma_r).
    \]

    \sloppy If $f$ is a continuous function on $[-1,1]^n$ define $g(\boldsymbol{\theta}) = f(\cos (\boldsymbol{\theta}) ) = f(\cos( \theta_1 ), \dots, \cos (\theta_n))$ for $\boldsymbol\theta \in [0,\pi]^n$. 
    Further, define for $\theta_i \in [-\pi,0]$ $g(\boldsymbol\theta) = g(\theta_1, \dots,-\theta_i, \dots, \theta_n)$. Similarly, we may define $g(\boldsymbol\theta)$ for all $\boldsymbol\theta \in [-\pi,\pi]^n$.
    We see $g(\boldsymbol\theta)$ is even and periodic on $[-\pi,\pi]^n$. Since it is even, the convolution with a kernel of the form \eqref{trigKernel} will have only cosine terms.
    The argument is as follows:
    \begin{align*}
       (2\pi)^n &\mathcal{B}^{(r)}(g)(\boldsymbol\theta)  = \int_{\textbf{K}}g(\boldsymbol\theta - \boldsymbol\varphi)K_r(\boldsymbol\varphi)\mathrm{d}\boldsymbol\varphi \\
       & = \int_{\textbf{B}}g(\boldsymbol\varphi)K_r(\boldsymbol\theta - \boldsymbol\varphi)\mathrm{d}\boldsymbol\varphi \\
       &= \int_{\textbf{B}}g(\boldsymbol\varphi)\left(1+\sum_{\alpha \in \mathbb{N}^n_r \setminus \{ 0 \}} 2^{H(\alpha)}g_\alpha \prod_{i = 1}^n \cos \left(\alpha_i (\theta_i - \varphi_i)\right) \right) \mathrm{d}\boldsymbol\varphi. \\
       &= \int_{\textbf{B}}g(\boldsymbol\varphi)\mathrm{d}\varphi + \\
       & \sum_{\alpha \in \mathbb{N}^n_r \setminus \{ 0 \}} 2^{H(\alpha)}g_\alpha \int_{\textbf{B}} g(\boldsymbol\varphi)\prod_{i = 1}^n \left(\cos (\alpha_i\theta_i) \cos (\alpha_i\varphi_i)  + \sin (\alpha_i \theta_i) \sin (\alpha_i \varphi_i) \right)\mathrm{d}\boldsymbol\varphi. \\ 
    \end{align*}
    The integrand in the last integral will be the function $g$ times the sum of products of sine and cosine functions.
    The domain $\mathbf{B} = [-\pi,\pi]^n$ is symmetric and $g$ is even by construction. Hence, every integral containing a sine function will evaluate to zero, since the sine function is odd. 
    We may now assume the approximation will take the following form 
    \begin{equation}
        q_r(\boldsymbol\theta) = \mathcal{K}^{(r)}(g)(\boldsymbol\theta) = a_0 + \sum_{\alpha \in \mathbb{N}^n_r} 2^{H(\alpha)}a_\alpha \prod_{i=1}^n \cos \left( \alpha_i \theta_i\right) \,.
    \end{equation}
    
    Substituting $\theta_i = \arccos (x_i) $ results in a polynomial
    \[
        p_r(\textbf{x}) = a_0 +   \sum_{\alpha \in \mathbb{N}^n_r}a_\alpha T_\alpha(\textbf{x}).
    \]

    This polynomial will serve as an approximation for $f$, and we will bound the absolute error in terms of $\sigma_r$. 
    For this we will need the following Lemma.
    \begin{lemma}\label{omegaLemma}
        For $f,g$ as defined above we have 
        \[
            \omega_g(\delta) \le \omega_f(\delta).
        \]
    \end{lemma}
    \begin{proof}
		The proof is straight-forward and omitted for the sake of brevity.
        % Note that
        % \begin{align*}
        %     \omega_g(\delta) &= \sup_{\stackrel{\boldsymbol\theta_1, \boldsymbol\theta_2 \in [-\pi, \pi]^n}{\|\boldsymbol\theta_1-\boldsymbol\theta_2 \|\le \delta}} |g(\boldsymbol\theta_1)-g(\boldsymbol\theta_2)| \\  
        %     &= \sup_{\stackrel{\boldsymbol\theta_1, \boldsymbol\theta_2 \in [0, \pi]^n}{\|\boldsymbol\theta_1-\boldsymbol\theta_2 \|\le \delta}} |f(\cos \boldsymbol\theta_1 )-f(\cos \boldsymbol\theta_2 )| \\
        %     & =   \sup_{\stackrel{\textbf{x}, \textbf{y} \in [-1, 1]^n}{\|\arccos \textbf{x} -\arccos \textbf{y} \|\le \delta}} |f(\textbf{x})-f(\textbf{y})| \\
        %     & \le \sup_{\stackrel{\textbf{x}, \textbf{y} \in [-1, 1]^n}{\|\textbf{x}- \textbf{y}\| \le \delta}} |f(\textbf{x})-f(\textbf{y})| = \omega_f(\delta), \\
        % \end{align*}
        % where we set $\arccos \textbf{x} = (\arccos x_1 ,\dots, \arccos x_n )$ to shorten the notation. Note that we used $|\cos x -\cos y | \le |x-y|$. 
    \end{proof}

    We have gathered everything required to prove \cref{prop1}.
    
    \begin{proof}[Proof of \cref{prop1}.]
      First, note that
  \begin{align*} 
      \sup_{\textbf{x} \in [-1,1]^n} |f(\textbf{x})-p_r(\textbf{x}) | & \le \sup_{\boldsymbol\theta \in [-\pi,\pi]^n} | g(\boldsymbol\theta)-q_r(\boldsymbol\theta) | \\
      & \le  2(1+\pi/\sqrt{2})\omega_g(\sigma_r)  \\
      & \le  2(1+\pi/\sqrt{2})\omega_f(\sigma_r),
  \end{align*}
  where the last inequality follows by \cref{omegaLemma}.
Above we have obtained a polynomial $p_r$ of degree less than $r$ which approximates $f \in \mathcal{C}(\mathbf{K})$. 
  We did so by using the convolution with a kernel of the form \eqref{trigKernel}. 
  Using the substitution given in \eqref{subst} we may transform the kernel into a positive kernel of the form
  \[
      K_r(\textbf{x},\textbf{y}) = 1+ \sum_{\alpha \in \mathbb{N}^n_r}2^{H(\alpha)}g_\alpha T_{\alpha}(\textbf{x})T_{\alpha}(\textbf{y}).  
  \]
  It is left to show that both approaches are equivalent, i.e., lead to the same approximation of the function $f$. For this note the following.
  The polynomial we obtain via the approximation process defined as per \eqref{convOperator} is 
  \begin{align*}
      \mathcal{K}^{(r)}(f)(\textbf{x})  & = \int_{\mathbf{K}}f(\textbf{y})\mathrm{d}\mu(\textbf{y}) +\sum_{\alpha \in \mathbb{N}^n_r}2^{H(\alpha)}g_\alpha \left(\int_{\mathbf{K}}f(\textbf{y})T_\alpha(\textbf{y})\mathrm{d}\mu(\textbf{y})\right) T_\alpha(\textbf{x}) \\
      & = c_0 + \sum_{\alpha \in \mathbb{N}^n_r} 2^{H(\alpha)} g_\alpha c_\alpha T_\alpha(\textbf{x}),
  \end{align*}
where 
\[
  c_\alpha = \langle f, T_\alpha \rangle_\mu = \int_{\mathbf{K}}f(\textbf{y})T_\alpha(\textbf{y})\mathrm{d}\mu(\textbf{y}).
\]
  We need to check whether we get the same coefficients from both approaches. 
  Recall the approximation from before
  \[
      p_r(\textbf{x}) = a_0 + \sum_{\alpha \in \mathbb{N}^n_r}2^{H(\alpha)}g_\alpha a_\alpha T_\alpha(\textbf{x}),
  \]
  where 
  \[
      a_\alpha = \frac{1}{(2\pi)^n} \int_{\mathbf{B}}g(\boldsymbol\varphi)\prod_{i=1}^n\cos (\alpha_i \varphi_i) \mathrm{d}\boldsymbol\varphi. 
  \]
We would like to show that $a_\alpha = c_\alpha$, i.e.,
  \[
      \int_{\mathbf{K}}f(\textbf{y})\prod_{i=1}^n\frac{\cos (\alpha_i \arccos (y_i) )}{\sqrt{1-y_i^2}}\mathrm{d}\textbf{y} = \frac{1}{2^n} \int_{\mathbf{B}}g(\boldsymbol\varphi)\prod_{i=1}^n\cos (\alpha_i \varphi_i) \mathrm{d}\boldsymbol\varphi.
  \]
  For this note that
  \begin{equation}\label{aalphaProof}
  a_\alpha = \frac{1}{2^n} \int_{\mathbf{B}}g(\boldsymbol\varphi)\prod_{i=1}^n\cos (\alpha_i \varphi_i) \mathrm{d}\boldsymbol\varphi = \int_{[0,\pi]^n}f(\cos \boldsymbol\varphi)\prod_{i=1}^n\cos (\alpha_i \varphi_i) \mathrm{d}\boldsymbol\varphi.		
\end{equation}

  Next we can make use of the following substitution 
  \[
      \varphi_i = \arccos (x_i)  \Rightarrow \mathrm{d}\varphi_i = -\frac{1}{\sqrt{1-x_i^2}}\mathrm{d}x_i. 
  \]
  Finally, using \eqref{aalphaProof} we find 
  \[
       a_\alpha =  \int_{[0,\pi]^n}f(\cos (\boldsymbol\varphi) )\prod_{i=1}^n\cos (\alpha_i \varphi_i) \mathrm{d}\boldsymbol\varphi = \int_{\mathbf{K}}f(\textbf{x})\prod_{i=1}^n\frac{\cos (\alpha_i \arccos (x_i ))}{\sqrt{1-x_i^2}}\mathrm{d}\textbf{x} = c_\alpha.
  \]
  This proves we indeed have that both approaches are equivalent, i.e., they lead to the same approximation. This completes the proof of \cref{prop1}.
\end{proof}

\subsection{Positivstellensatz for the multivariate case}
\label{sec:multivariate KPM}
In this subsection we present a Positivstellensatz for multivariate trigonometric polynomials. This result allows for the construction of semidefinite programs whose solutions provide optimal kernels with respect to $\sigma_r$. 
To this end, recall the identities:
if $p_k(\phi) = \cos (k\phi)$, then, for $x,y \in [-1,1]$,
\begin{eqnarray*}
  p_k(\arccos (x)) &=& T_k(x) \\
  \frac{1}{2}\left(p_k(\arccos (x) + \arccos (y)) + p_k(\arccos (x) - \arccos (y))\right) &=& T_k(x)T_k(y).
\end{eqnarray*}
As a consequence, if we start with  a non-negative multivariate trigonometric polynomial of the form
\[
(\phi_1,\ldots, \phi_n) \mapsto 1+ \sum_{\alpha \in \mathbb{N}^n_r\setminus \{0\}} 2^{H(\alpha)} g_\alpha \prod_{ i \in [n]} \cos (\alpha_i \phi_i),
\]

then replacing each $p_{\alpha_i}(\phi_i) := \cos (\alpha_i \phi_i)$  by
$$\frac{1}{2}\left(p_{\alpha_i}(\arccos (x_i) + \arccos (y_i)) + p_{\alpha_i}(\arccos (x_i) - \arccos (y_i))\right)$$
as in \eqref{subst} (this operation preserves non-negativity), one obtains the non-negative kernel
\[
K(\textbf{x},\textbf{y}) = 1 + \sum_{\alpha \in \mathbb{N}^n_r\setminus \{0\}} 2^{H(\alpha)}g_\alpha \prod_{ i \in [n]} T_{\alpha_i}(x_i)T_{\alpha_i}(y_i) \quad \textbf{x},\textbf{y} \in \mathbf{K}.
\]
In contrast to the polynomial case, each multivariate, positive, trigonometric polynomial is a sum of squares of trigonometric polynomials. (The degrees appearing in the sums-of-squares may be arbitrarily large, though.)

\begin{theorem}[e.g.\ Theorem 3.5 in \cite{Dumitrescu2007}]\label{ThmTrigSDP}
  If $p$ is a positive trigonometric polynomial, then there exists an $r \in \mathbb{N}$ and a hermitian p.s.d. matrix $M$ of order $\binom{n+r}{r}$ such that
  \[
  p(\boldsymbol\phi) = \left[ \exp (\imath \alpha^T\boldsymbol\phi)\right]_{\alpha \in \mathbb{N}^n_r}^*M\left[ \exp (\imath \alpha^T\boldsymbol\phi)\right]_{\alpha \in \mathbb{N}^n_r}.
  \]
\end{theorem}
Again, the value $r$ in the theorem may be arbitrarily large, but for fixed $r$, one may consider the kernel given by solving the SDP:

\begin{equation}
	\label{trig kernel SDP}
	 \sigma^2_r = \min_{g_\alpha \,:\, \alpha \in \mathbb{N}^n_r} \;\sum_{i=1}^n \left( 1 - g_{e_i} \right)
	\end{equation}
	subject to

  \begin{eqnarray*}
		%t_i &\ge& (1-2g_{e_i}+g_{2e_i}) \quad \forall i \in [n] \\
		%t_i &\ge& -(1-2g_{e_i}+g_{2e_i}) \quad \forall i \in [n] \\
		1+ \sum_{\alpha \in \mathbb{N}^n_r\setminus \{0\}} 2^{H(\alpha)}g_\alpha \prod_{ i \in [n]} \cos (\alpha_i \phi_i) &=& \left[ \exp (\imath \alpha^T\boldsymbol\phi)\right]_{\alpha \in \mathbb{N}^n_r}^*M\left[ \exp (\imath \alpha^T\boldsymbol\phi)\right]_{\alpha \in \mathbb{N}^n_r} \\
		M &\succeq & 0.
	\end{eqnarray*}

  \begin{figure}[tbhp]
    \centering 
    \subfloat[$r=30$]{\label{fig:a}\includegraphics[width=0.48\textwidth]{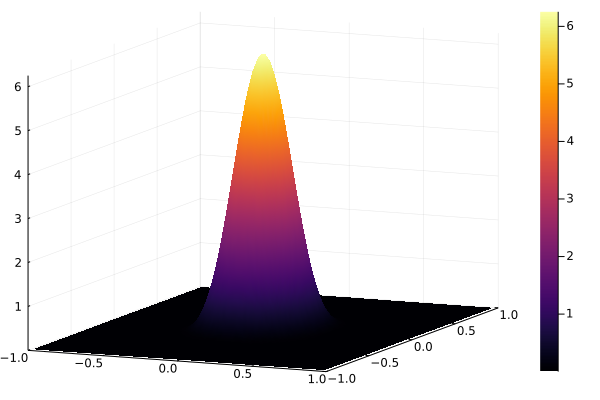}} 
    \subfloat[$r=50$]{\label{fig:b}\includegraphics[width=0.48\textwidth]{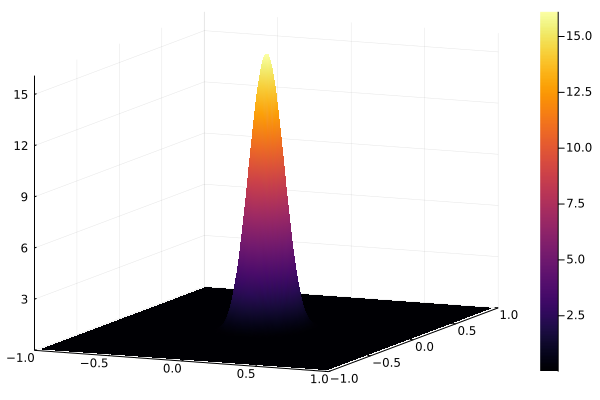}} 
    \caption{Plots of $\mathcal{K}^{(r)}(\delta_{(0,0)})$, i.e. convolution of kernel with minimum resolution $\sigma_r$ with the Dirac-$\delta$ at $(x,y) = (0,0)$ for different values of $r$}
    \label{fig:testfig}
    \end{figure}

  \section{Reformulation of the SDP}
	\label{sec:implementation}
	In this section we present how to find solutions to problem \eqref{trig kernel SDP}. But first we show that the optimal solution to the problem maybe assumed w.l.o.g. to be real. 
	\subsection{Existence of a real solution}
	Let $M = P+ \imath Q$, for $P \in \mathbb{S}^{s(n,r)}$ and $Q$ being skew-symmetric of the same size. Then,
	\begin{align*}
		\left[ \exp (\imath \alpha^T\boldsymbol\phi)\right]_{\alpha \in \mathbb{N}^n_r}^*(P+\imath Q)\left[ \exp (\imath \alpha^T\boldsymbol\phi)\right]_{\alpha \in \mathbb{N}^n_r} = & \sum_{\alpha, \beta \in \mathbb{N}_r^n} \cos\left[(\alpha-\beta)^T\boldsymbol\phi\right] P_{\alpha,\beta} \\
		& - \sum_{\alpha , \beta \in \mathbb{N}_r^n} \sin \left[(\alpha-\beta)^T \boldsymbol\phi \right] Q_{\alpha,\beta},
	\end{align*}
	because of the identities
	\[
	\cos \left(\alpha^T\boldsymbol\phi\right) \cos \left(\beta^T\boldsymbol\phi\right) +\sin \left(\alpha^T\boldsymbol\phi \right) \sin \left(\beta^T\boldsymbol\phi\right)  = \cos\left[(\alpha-\beta)^T\boldsymbol\phi\right]
	\]
	and
	\[
	\cos \left(\alpha^T\boldsymbol\phi\right) \sin \left(\beta^T\boldsymbol\phi \right)-\sin \left(\alpha^T\boldsymbol\phi \right)\cos \left(\beta^T\boldsymbol\phi\right)  = \sin\left[(\alpha-\beta)^T\boldsymbol\phi\right].
	\]
	
	We continue to show we may without loss of generality assume there exists an optimal solution which is real symmetric, i.e., $Q = 0$. Let $P+\imath Q$ be an optimal solution to \eqref{trig kernel SDP} and define 
	\begin{align*}
	G(\phi_1, \dots,\phi_n) &= \sum_{\gamma \in \mathbb{N}^n_r} 2^{H(\gamma)}g_\gamma \prod_{ i \in [n]} \cos \left( \gamma_i \phi_i \right) \\
	&= \sum_{\alpha, \beta \in \mathbb{N}_r^n} \cos\left[(\alpha-\beta)^T\boldsymbol\phi\right] P_{\alpha,\beta} 
	 - \sum_{\alpha , \beta \in \mathbb{N}_r^n} \sin \left[(\alpha-\beta)^T \boldsymbol\phi \right] Q_{\alpha,\beta}
	\end{align*}
	
	Now, since the cosine is even we find $G(\phi_1, \dots,\phi_n) = G(-\phi_1, \dots,-\phi_n)$. So
	\begin{align*}
	G(\phi_1, \dots,\phi_n) &= \frac{1}{2}(G(\phi_1, \dots,\phi_n) + G(-\phi_1, \dots,-\phi_n)) \\
	&= \frac{1}{2}\Big( \sum_{\alpha, \beta \in \mathbb{N}_r^n} \cos\left[(\alpha-\beta)^T\boldsymbol\phi\right] P_{\alpha,\beta}
	 - \sum_{\alpha , \beta \in \mathbb{N}_r^n} \sin \left[(\alpha-\beta)^T \boldsymbol\phi \right] Q_{\alpha,\beta} \\
	 &\quad + \sum_{\alpha, \beta \in \mathbb{N}_r^n} \cos\left[(\beta-\alpha)^T\boldsymbol\phi\right] P_{\alpha,\beta}
	 - \sum_{\alpha , \beta \in \mathbb{N}_r^n} \sin \left[(\beta-\alpha)^T\boldsymbol\phi \right] Q_{\alpha,\beta}\Big) \\
	 & = \sum_{\alpha, \beta \in \mathbb{N}_r^n} \cos\left[(\alpha-\beta)^T\boldsymbol\phi\right] P_{\alpha,\beta}.
	\end{align*}
	
	Therefore, we may subsequently assume that in \eqref{trig kernel SDP} the matrix $M$ is real symmetric.

	\subsection{Equating coefficients via trigonometric identities}
	The standard idea of SOS-optimization is equating coefficients. For convenience, we restate the optimization problem. Recall that it is enough to consider real symmetric matrices.
	
	\begin{equation}\label{trig kernel SDP2}
	\sigma^2_r = \min_{g_\alpha \,:\, \alpha \in \mathbb{N}^n_r} \;\sum_{i=1}^n \left(1 - g_{e_i} \right)
	\end{equation}
	subject to
	\begin{eqnarray}
	%t_i &\ge& (1-2g_{e_i}+g_{2e_i}) \quad \forall i \in [n] \\
	%t_i &\ge& -(1-2g_{e_i}+g_{2e_i}) \quad \forall i \in [n] \\
	1+ \sum_{\gamma \in \mathbb{N}^n_r\setminus \{0\}} 2^{H(\gamma)}g_\gamma \prod_{ i \in [n]} \cos \gamma_i \phi_i &=& \sum_{\alpha,\beta \in \mathbb{N}^n_r}\cos[(\alpha-\beta)^T\boldsymbol\phi]M_{\alpha,\beta} \label{equality} \\
	M &\succeq & 0, \nonumber %\; M \in \mathbb{R}^{s(n,r)\times s(n,r)}.
	\end{eqnarray}

	However, the form in which the two representations of the sought polynomial are given does not allow for an immediate construction of the corresponding constraint matrices. The question is, what $\alpha, \beta$ of the right-hand-side end up contributing to a $\gamma$ of the left-hand-side, and in what way. 
	Let for $I \subseteq [n]$ the function $\omega_{I} : \mathbb{R}^n \to \mathbb{R}^n$ be defined as follows
	\[
	\omega_I(\textbf{x})_i = \begin{cases} -x_i, \text{ if } i \in I\\
	x_i, \text{ otherwise}.\end{cases}
	\]
	In other words, $\omega_I$ flips the sign of $x_i$ for all $i \in I$.  Recall the trigonometric identity
	\begin{equation}
	\label{trig:id}
	 2^{H(\textbf{x})} \prod_{i=1}^n \cos (x_i)  = \sum_{I \subseteq [n]} \cos \left[\sum_{i = 1}^n \omega_I(\textbf{x})_i \right],
	\end{equation}
	which may be proved by induction on $n$ by using the well-known identity 
	\[
	\cos (x+y)  = \cos (x) \cos (y) -\sin (x) \sin (y) \,.
	\]
	The identity \eqref{trig:id}
	will allow us to compare coefficients of trigonometric polynomials in
	\eqref{trig kernel SDP2}. On the right-hand side in \eqref{equality} we will find all $\gamma \in \mathbb{Z}^n$ for which there exist $\alpha, \beta \in \mathbb{N}^n_r$, such that $\gamma = \alpha - \beta$. The identity \eqref{trig:id} now tells us that we have to make sure that for a given $\gamma \in \mathbb{N}^n_r$ the following holds 
	\[
	\sum_{\stackrel{\alpha, \beta \in \mathbb{N}^n_r}{\alpha-\beta = \gamma}} M_{\alpha, \beta} = \sum_{\stackrel{\alpha, \beta \in \mathbb{N}^n_r}{\alpha-\beta = \omega_I(\gamma)}} M_{\alpha, \beta} \quad \text{ for all } I \subset [n],
	\]
	since then we can factor out the same sum for each $\omega_I(\gamma)$ and apply identity \eqref{trig:id}. Noting that $\alpha-\beta = -(\beta-\alpha)$ we can construct symmetric constraint matrices. For each $\gamma \in \mathbb{N}^n_r$ let $C^{(\gamma, I)}\in \{0,1\}^{s(n,r)\times s(n,r)}$
	\[
	C^{(\gamma, I)}_{\alpha, \beta} =  \begin{cases} 1, \text{ if } \alpha - \beta = \omega_I(\gamma) \lor \omega_{I^c}(\gamma) \\
	0, \text{ otherwise.}
	\end{cases}
	\]
	These matrices will always be symmetric since if $\alpha-\beta = \omega_I(\gamma)$ then $\beta - \alpha = \omega_{I^c}(\gamma)$. We define $\mathcal{I}$ as a set of subsets of $[n]$ such that no complement $I^c$ of a set $I \in \mathcal{I}$ lies in $\mathcal{I}$ and $\cup_{I \in \mathcal{I}} \{ I , I^c \} = \{ I : I \subseteq [n] \}$.
	With this we can formulate the first set of constraints, i.e.,
	\[
	\langle M, C^{(\gamma,\emptyset)} \rangle = \langle M, C^{(\gamma,I)} \rangle \; \forall I \in \mathcal{I}, \forall \gamma \in \mathbb{N}^n_r.
	\]
	
	Then $g_{\gamma} = \frac{1}{2} \langle M, C^{(\gamma,I)} \rangle$ for any $I \in \mathcal{I}$. Additionally, we need the following. Let 
	\[
	\Gamma_{(n,r)} = \left\{ \gamma \in \mathbb{Z}^n : \exists \alpha, \beta \in \mathbb{N}^n_r\,, \alpha - \beta = \gamma\, \land \, \sum_{i = 1}^n |\gamma_i| > r \right\}, 
	\]
	which leads us to the next set of constraints
	\[
	\langle M, C^{(\gamma,I)} \rangle = 0 \; \forall I \in \mathcal{I}, \forall \gamma \in \Gamma_{(n,r)}.
	\]
	This way we ensure that there will not appear any unwanted terms in the resulting polynomial. Any $\gamma \in \Gamma_{(n,r)}$ will not find the necessary pairs to use identity \eqref{trig:id}. Therefore, we force all such terms to be zero. We can formulate the SDP now as 
	\begin{equation}\label{sdptrigIdent}
	 \sigma^2_r = \min \;\sum_{i=1}^n \left( 1 - \frac{1}{2}\langle M, C^{(e_i, \emptyset)} \rangle \right)
	\end{equation}
	subject to
	\begin{eqnarray*}
	%	t_i &\ge& \left(1-\langle M, C^{(e_i, \emptyset)}+\frac{1}{2}C^{(2e_i, \emptyset)} \rangle \right) \quad \forall i \in [n] \\
	%	t_i &\ge& -\left(1-\langle M, C^{(e_i, \emptyset)}+\frac{1}{2}C^{(2e_i, \emptyset)} \rangle \right) \quad \forall i \in [n] \\
		\langle M, C^{(\alpha,\emptyset)}- C^{(\alpha,I)} \rangle &=& 0 \quad \forall I \in \mathcal{I}, \forall \alpha \in \mathbb{N}^n_r\setminus \{ 0 \} \\
		\langle M, C^{(\gamma,I)} \rangle &=& 0 \quad \forall I \in \mathcal{I}, \forall \gamma \in \Gamma_{(n,r)} \\
		\mathrm{Tr}(M) &=& 1\\
		M &\succeq & 0.
	\end{eqnarray*}
	Note that $\mathrm{Tr}(M)=1$ ensures $g_\alpha = 1$ for $\alpha = (0,\dots,0)$.
	% This implementation works better in practice than the previous one \eqref{samplingSDP}, especially for larger values of $r$. 
	
  \section{Symmetry reduction}
	\label{sec:symmetry}
	In this section we will present an approach that exploits existing symmetries in semidefinite programs in order to improve numerical tractability. There has been done research on the exploitation of symmetries in semidefinite programming \cite{Vallentin2009}. There are also results available focusing on symmetry exploitation for semidefinite relaxations of polynomial optimization problems. We refer the reader to \cite{Gatermann2004}, \cite{Riener2013}. The name readily implies that we will use some \emph{symmetry} to \emph{reduce} the size of the SDP. The goal is to set up an equivalent SDP, where we can impose a block diagonal structure on the matrix variable. This is helpful because we then only have to enforce the positive semidefiniteness for the individual blocks instead of the whole matrix. 

  \subsection{Symmetry adapted basis}
	Let $\mathcal{S}_n$ be the symmetric group acting on the variables $x_i$ for $i \in [n]$ by permuting the elements, i.e.,
	\[
	\sigma(x_i) = x_{\sigma(i)} \text{ for } i \in [n], \sigma \in \mathcal{S}_n.
	\]
	
	The action of $\mathcal{S}_n$ may be defined on functions as well in the following way. Let $f : \mathbb{R}^n \rightarrow \mathbb{R}$, then 
	\[
	\sigma(f) = f(\sigma(\textbf{x})),
	\]
	where, if an element $\sigma \in \mathcal{S}_n$ is applied to an $n$-tuple, we define for $\textbf{x} \in \mathbb{R}^n$ 
	
	\[ \sigma(\textbf{x}) = (\sigma(x_1), \sigma(x_2)\dots, \sigma(x_n)) = (x_{\sigma(1)}, x_{\sigma(2)}, \dots, x_{\sigma(n)}), \] 
	i.e., elementwise application.
	We will call a function $f$ invariant under $\mathcal{S}_n$ if $\sigma(f) = f(\textbf{x})$ for all $\sigma \in \mathcal{S}_n$. Note that our kernel is defined over the set $[-\pi,\pi]^n$, which is invariant under the action of $\mathcal{S}_n$. We can also assume without loss of generality that the optimal kernel $K_r$ will be invariant under the action of $\mathcal{S}_n$, meaning that for all coefficients we will have
	\[
	g_\alpha = g_{\sigma(\alpha)} \text{ for all } \sigma \in \mathcal{S}_n, \alpha \in \mathbb{N}^n_r,
	\]

	To see that the optimal kernel will be invariant under $\mathcal{S}_n$, note that all constraints in problem \eqref{trig kernel SDP2} are invariant under $\mathcal{S}_n$. Thus, any optimal solution to \eqref{trig kernel SDP2} can be "symmetrized" using the Reynolds-operator, which is defined as 
	\[
	\mathcal{R}^{\mathcal{S}_n}(f,g) := \frac{1}{|\mathcal{S}_n |} \sum_{\sigma \in \mathcal{S}_n} \sigma (f) \sigma (g).
	\]
	Let $K_r = \sum_{\alpha \in \mathbb{N}^n_r} \tilde{g}_\alpha\prod_{ i \in [n]}\cos \alpha_i \phi_i $ be a feasible solution to \eqref{trig kernel SDP2}, then 
	\[
	\mathcal{R}^{\mathcal{S}_n}(K_r,1) =  \frac{1}{|\mathcal{S}_n|} \sum_{\sigma \in \mathcal{S}_n} \sum_{\alpha \in \mathbb{N}^n_r} \tilde{g}_\alpha\sigma \left(\prod_{ i \in [n]}\cos (\alpha_i \phi_i) \right)
	\]
	is also feasible and will lead to the same objective value. Therefore, we may assume the optimal kernel is invariant under $\mathcal{S}_n$, as well. 
	
	We will continue to summarize how to use the symmetry reduction technique described in \cite[Section 2.4, 4.1]{Riener2013}, without dwelling on details unnecessary for our discussion. 
	Therefore, we will not focus on the technicalities of the construction of the \emph{symmetry adapted basis}, but rather show how it may be used. 
	The idea is as follows. Let $\mathbb{T}[\varphi]_r = \mathbb{T}[\varphi_1, \dots, \varphi_n]_{r}$ be the set of trigonometric polynomials of degree less that $r$. 
	We will define $\mathbb{T}[\varphi]^{\mathcal{S}_n}_{r}$ to be the set of trigonometric polynomials of degree at most $r$ which are invariant under the action of $\mathcal{S}_n$. A basis for $\mathbb{T}[\varphi]_r$ is given by $\{ \exp (\imath \alpha^T\phi) \}_{\alpha \in \mathbb{N}^n_r}$. 
	To exploit the symmetry we will construct a new basis $\mathcal{B}$, which we call the \emph{symmetry adapted basis}. 
	The set $\mathcal{B}$ may be seen as a collection of $k(n,r) \in \mathbb{N}$ sub-bases $\mathcal{B}_i = \{ b_{i_j} \in \mathbb{T}[\varphi]_r  : \text{ for } j \in [k_i] \}$, for some $k_i \in \mathbb{N}$ in the sense that $\mathcal{B} = \{\mathcal{B}_i : i \in [k(n,r)]\}$. 
	In general there are no closed form expressions for $k(n,r)$ and $k_i$ available as functions of $n$ and $r$, but to give some impression of these numbers we provide \cref{comparisonSymRed}. 
	We call $\mathcal{B}$ a basis because
	\[
	\textrm{span}\left\{ \mathcal{R}^{\mathcal{S}_n}(b_{i_l}, b_{i_m}^\ast),\, i\in [k(n,r)],\, l,m \in [k_i] \right\} = \mathbb{T}[\varphi]_r^{\mathcal{S}^n}.
	\]
	The basis $\mathcal{B}$ has the property that its elements are pairwise orthogonal in the sense that for $b_{i_l} \in \mathcal{B}_i, b_{j_m} \in \mathcal{B}_j$ with $i \neq j$ the symmetrized product is zero, i.e.,
	\[
	\mathcal{R}^{\mathcal{S}_n}(b_{i_l},b_{j_m}^\ast) = 0.
	\]
	
	Before, we were interested in suitable kernels that could be written as 
	\[
	K_r = \left[ \exp (\imath \alpha^T\phi)\right]_{\alpha \in \mathbb{N}^n_r}^*M\left[ \exp (\imath \alpha^T\phi)\right]_{\alpha \in \mathbb{N}^n_r},
	\]
	where $M \in \mathbb{S}^{s(n,r)}_{\succeq 0}$. Knowing that the optimal $K_r$ is invariant under $\mathcal{S}_n$, we can write 
	\[
	K_r = \sum_{i = 1}^{k(n,r)} [b_{i,j}]_{j \in [k_i]}M^{(i)}[b_{i,j}]_{j \in [k_i]}
	\]
	for $M^{(i)}\succeq 0$ for all $i \in [k(n,r)]$. The pairwise orthogonality of $\mathcal{B}$ means that we can consider a block diagonal matrix
	\[
	\begin{bmatrix}
	M^{(1)} & 0 & 0 & \dots & 0 \\
	0 & M^{(2)} & 0 & \dots & 0 \\
	0 & 0 & \ddots & \ddots  & \vdots \\
	\vdots & \vdots & \ddots & \ddots & 0 \\
	0 & 0 & \dots & 0 & M^{(k(n,r))}
	\end{bmatrix}
	\]
	with $M^{(i)} \in \mathbb{S}^{k_i}_{\succeq 0 }$ in our SDP. The computational advantage is that we only have to ensure the positive semidefiniteness of the individual blocks, instead of the much larger matrix $M$. 
	
\begin{example}
  Consider the following example for a symmetry adapted basis with $n=2,r=2$. In this case $k(n,r) = k(2,2) = 2$ and $k_1 = 4, k_2 = 2$. 
	For a corresponding SDP we would have to consider two psd blocks of sizes $4$ and $2$ instead of one psd matrix of size $6\times6$. 
	For $\mathcal{B}_1, \mathcal{B}_2$ we find
	%\begin{minipage}[t]{0.5\textwidth}
		\begin{align*}
			\mathcal{B}_1 = \big \{ & \exp(\imath\,(0 \varphi_1+0 \varphi_2))  = 1, \\
			& \exp(\imath\,(\varphi_1+ \varphi_2)), \\
			& \exp(\imath\, \varphi_1) + \exp(\imath\, \varphi_2), \\ 
			& \exp(\imath \, 2 \varphi_1) + \exp(\imath \, 2 \varphi_2) \big \}
		\end{align*}
	%\end{minipage}%
	%
%	\begin{minipage}[t]{0.5\textwidth}
		\begin{align*}
			\mathcal{B}_2 = \big \{ 
			& \exp(\imath\, \varphi_1) - \exp(\imath\, \varphi_2), \\ 
			& \exp(\imath \, 2 \varphi_1) - \exp(\imath \, 2 \varphi_2) \big \}.
		\end{align*}
  \end{example}

\subsection{Construction of the symmetry adapted SDP}
	For $\alpha \in \mathbb{N}^n_r$ we define the corresponding \emph{orbit} as $\mathcal{O}_\alpha = \{ \sigma(\alpha) :  \sigma \in \mathcal{S}_n \}$. For each orbit we choose a representative $\alpha$ which is sorted, i.e., $\alpha_1 \le \alpha_2 \le \dots \le \alpha_n$ and index the orbit by that $\alpha$. We define $S(\mathbb{N}^n_r) = \{ \alpha \in \mathbb{N}^n_r : \alpha_1 \le \alpha_2 \le \dots \le \alpha_n \}$ to be the set of representatives for the set of orbits. The set $\mathbb{N}^n_r$ can be written as the union of orbits, i.e.,
	\[
	\mathbb{N}^n_r = \bigcup\limits_{\alpha \in S(\mathbb{N}^n_r)} \mathcal{O}_\alpha.
	\]
	The invariance of $K_r$ means that $g_\alpha = g_\beta$ for every $\beta \in \mathcal{O}_\alpha$. We are now equipped to reformulate \eqref{trig kernel SDP2} as an equivalent optimization problem which is easier to solve. We first note that for the invariant kernel we have
	\begin{equation*}
		\begin{aligned}
			\sum_{\alpha \in \mathbb{N}^n_r} 2^{H(\alpha)}g_\alpha \prod_{i = 1}^n \cos (\alpha_i \varphi_i) & = \sum_{\alpha \in S(\mathbb{N}^n_r)} 2^{H(\alpha)} g_\alpha \left( \sum_{\beta \in \mathcal{O}_\alpha} \prod_{i = 1}^n \cos (\beta_i \varphi_i)  \right).
		\end{aligned}
	\end{equation*}
	
	For every $\mathcal{B}_i \in \mathcal{B}$ we define a matrix $M^{(i)}$ of size $k_i \times k_i$ with $k_i = |\mathcal{B}_i|$. Then the program may be written as follows.
	
	\begin{equation}
	\sigma^2_r = \min \; n \left(1 - g_{e_n}\right)
	\end{equation}
	subject to
	\begin{eqnarray*}
		\sum_{\alpha \in S(\mathbb{N}^n_r)} 2^{H(\alpha)} g_\alpha \left( \sum_{\beta \in \mathcal{O}_\alpha} \prod_{i = 1}^n \cos (\beta_i \phi_i)  \right) &=& \sum_{i=1}^{k(n,r)} \sum_{j,\ell=1}^{k_i} \mathcal{R}^{\mathcal{S}_n}(b_{i_j},b_{i_\ell}^\ast) M^{(i)}_{j,\ell} \\
		M^{(i)} &\succeq & 0, \; i \in [k].
	\end{eqnarray*}
	
	Let us take a closer look at the terms $\mathcal{R}^{\mathcal{S}_n}(b_{i_l},b_{i_m}^\ast)$. Assume that the elements $b_{i_j} \in \mathcal{B}_i$ are given in the following form 
	\[
	b_{i_j} = \sum_{m = 1}^{k_{i_j}} b_{i_j}^{(m)} \exp\left(\imath \left(\alpha^{(m)}\right)^T \boldsymbol\phi \right).
	\]
	
	Then, 
	\begin{equation*}
		\begin{aligned}
			\mathcal{R}^{\mathcal{S}_n}(b_{i_j},b_{i_\ell}^\ast)
      & = \frac{1}{|\mathcal{S}_n|}\sum_{\sigma \in \mathcal{S}_n} \sigma(b_{i_j})\sigma(b_{i_\ell}^\ast) \\
			&= \frac{1}{|\mathcal{S}_n|}\sum_{\sigma \in \mathcal{S}_n} \sum_{m = 1}^{k_{i_j}}\sum_{p = 1}^{k_{i_\ell}}b_{i_j}^{(m)} b_{i_\ell}^{(p)}\exp\left(\imath \left(\alpha^{(m)}\right)^T \sigma(\boldsymbol\phi) +\imath \left(\beta^{(p)}\right)^T \sigma(\boldsymbol\phi) \right)^\ast \\
			& = \frac{1}{|\mathcal{S}_n|}\sum_{\sigma \in \mathcal{S}_n} \sum_{m = 1}^{k_{i_j}}\sum_{p = 1}^{k_{i_\ell}}b_{i_j}^{(m)} b_{i_\ell}^{(p)} \cos[\sigma(\alpha^{(m)}-\beta^{(p)})^T \boldsymbol\phi].
		\end{aligned}
	\end{equation*}
	
	Recalling the trigonometric identity \eqref{trig:id} from before, we can now construct the constraint matrices. Let $\gamma \in S(\mathbb{N}^n_r)$. For each $i \in [k(n,r)], I \in \mathcal{I}$ for $\mathcal{I}$ as in the previous subsection we define
	\[
	\left(C_i^{(\gamma, I)}\right)_{j, \ell} = \begin{cases}
	c(\gamma, I, i, j, \ell)), \text{ if } \omega_I(\gamma) \text{ or } \omega_{I^c}(\gamma) \text{ occurs in } 	\mathcal{R}^{\mathcal{S}_n}(b_{i_j},b_{i_\ell}^\ast) \\
	0, \text{ otherwise},
	\end{cases}
	\]
	where 
	\[
	c(\gamma, I, i, j, \ell))= \frac{1}{|\mathcal{S}_n |} \sum_{\stackrel{m,p \in [k_i]}{\alpha^{(m)}-\beta^{(p)} = \pm \omega_I(\gamma)}} b_{i_j}^{(m)}b_{i_\ell}^{(p)}.
	\]
	As before, we must ensure that for all $I \in \mathcal{I}$ we have that the corresponding coefficients in our resulting polynomial are equal. Therefore, we arrive at the set of constraints 
	\[
	\sum_{i = 1}^{k(n,r)} \langle M^{(i)}, C^{(\gamma, \emptyset)}_i - C^{(\gamma, I)}_i \rangle = 0\,, \text{ for every } \gamma \in S(\mathbb{N}^n_r), I \in \mathcal{I}, I \neq \emptyset.  
	\]
	We will now define the set $S\Gamma_{(n,r)} = \Gamma_{(n,r)} / \mathcal{S}_n$, which is the "symmetry adapted" version of $\Gamma_{(n,r)}$ of the previous subsection, where we factored out all permutations $\sigma \in \mathcal{S}_n$ of a reference element $\gamma$ except the identity. This leads to the constraint set
	\[
	\sum_{i = 1}^{k(n,r)} \langle M^{(i)}, C^{(\gamma, I)}_i \rangle = 0\,, \text{ for every } \gamma \in S\Gamma_{(n,r)}, I \in \mathcal{I}.  
	\]
	The resulting SDP reads as follows. 
	
	\begin{equation}\label{sdpsymred}
	\sigma^2_r = \min \;n \left( 1 - \frac{1}{2} \sum_{i=1}^{k(n,r)}\langle M^{(i)}, C^{(e_n, \emptyset)}_i \rangle \right)
	\end{equation}
	subject to
	\begin{eqnarray*}
		\sum_{i = 1}^{k(n,r)} \langle M^{(i)}, C^{(\gamma, \emptyset)}_i - C^{(\gamma, I)}_i \rangle &=& 0\,, \text{ for every } \gamma \in S(\mathbb{N}^n_r) / \{(0,\dots,0)\}, I \in \mathcal{I}, I \neq \emptyset \\
		\sum_{i = 1}^{k(n,r)} \langle M^{(i)}, C^{(\gamma, I)}_i \rangle &=& 0\,, \text{ for every } \gamma \in S\Gamma_{(n,r)}, I \in \mathcal{I} \\
		\sum_{i = 1}^{k(n,r)} \langle M^{(i)}, C_i^{((0,\dots,0),\emptyset)} \rangle &=& 1\\
		M^{(i)} &\succeq & 0\; \text{ for all } i \in [k(n,r)].
	\end{eqnarray*}
	
	The efficiency of using this symmetry reduction increases when $n$ grows, since the underlying group is $\mathcal{S}_n$. 
	Fixing $n$ and increasing the degree $r$ the size of the underlying program still grows fast. 
	It is also possible to use software for the symmetry reduction, such as the Julia package SDPSymmetryReduction.jl\footnote{see https://github.com/DanielBrosch/SDPSymmetryReduction.jl} which is based on the paper \cite{Brosch2022} by Brosch and de Klerk. 
	This software takes as input a semidefinite program and numerically performs a symmetry reduction without any need to specify the underlying group. 
	The advantage that comes with this is that there is no need for the construction of a specific symmetry adapted basis. 
	But even for small $n$, if $r$ becomes too large the resulting optimization problem becomes numerically unstable. 
	It is still worthwhile to compare the two approaches. 
	The block sizes which are returned by the software are the same as the ones we obtained by our approach presented in this section.
	This suggests that the symmetry is fully exhausted by the symmetric group $\mathcal{S}_n$. 

	To give some insights into the efficiency of the symmetry reduction, we present in \cref{comparisonSymRed} a comparison of the sizes and numbers of the blocks of the symmetry reduced program vs. the size of the SDP matrix in the case without symmetry reduction, i.e. program \eqref{sdptrigIdent}. 
	
	\begin{table}
		\centering
		\scalebox{1.0}{
			\begin{tabular}{ |c|c||c|c|c| }
			\hline
			$n$& $r$& $k(n,r)$ & $k_1, \dots, k_{k(n,r)}$ & $s(n,r)$ \\
			\hline
			\hline
			2 & 1& 2&2 , 1 ,  & 3 \\ 
\hline 
$2$ & $1$& $2$&$2$, $1$   & $3$ \\ 
\hline
$2$ & $2$& $2$&$4$, $2$   & $6$ \\ 
\hline 
$2$ & $3$& $2$&$6$, $4$   & $10$ \\ 
\hline
$2$ & $4$& $2$&$9$, $6$   & $15$ \\ 
\hline 
$2$ & $5$& $2$&$12$, $9$   & $21$ \\ 
\hline 
$2$ & $6$& $2$&$16$, $12$   & $28$ \\ 
\hline 
$2$ & $7$& $2$&$20$, $16$   & $36$ \\ 
\hline 
$2$ & $8$& $2$&$25$, $20$  & $45$ \\ 
\hline
$2$ & $9$& $2$&$30$, $25$   & $55$ \\ 
\hline 
$2$ & $10$& $2$&$36$, $30$   & $66$ \\ 
\hline 
\hline 
$3$ & $1$& $2$&$2$, $1$  & $4$ \\ 
\hline 
$3$ & $2$& $2$&$4$, $3$  & $10$ \\ 
\hline 
$3$ & $3$& $3$&$7$, $6$, $1$   & $20$ \\ 
\hline 
$3$ & $4$& $3$&$11$, $11$, $2$   & $35$ \\ 
\hline 
$3$ & $5$& $3$&$16$, $18$, $4$  & $56$ \\ 
\hline 
$3$ & $6$& $3$&$23$, $27$, $7$  & $84$ \\ 
\hline 
$3$ & $7$& $3$&$31$, $39$, $11$  & $120$ \\ 
\hline 
$3$ & $8$& $3$&$41$, $54$, $16$   & $165$ \\ 
\hline 
$3$ & $9$& $3$&$53$, $72$, $23$   & $220$ \\ 
\hline 
$3$ & $10$& $3$&$67$, $94$, $31$   & $286$ \\ 
\hline 
\hline 
$4$ & $1$& $2$&$2$, $1$  & $5$ \\ 
\hline 
$4$ & $2$& $3$&$4$, $3$, $1$   & $15$ \\ 
\hline 
$4$ & $3$& $4$&$7$, $7$, $2$, $1$   & $35$ \\ 
\hline 
$4$ & $4$& $4$&$12$, $13$, $5$, $3$   & $70$ \\ 
\hline 
$4$ & $5$& $4$&$18$, $23$, $9$, $7$  & $126$ \\ 
\hline 
$4$ & $6$& $5$&$27$, $37$, $16$, $13$, $1$  & $210$ \\ 
\hline 
$4$ & $7$& $5$&$38$, $57$, $25$, $23$, $2$   & $330$ \\ 
\hline 
$4$ & $8$& $5$&$53$, $83$, $39$, $37$, $4$   & $495$ \\ 
\hline 
$4$ & $9$& $5$&$71$, $118$, $56$, $57$, $7$  & $715$ \\ 
\hline
$4$ & $10$& $5$&$94$, $162$, $80$, $83$, $12$  & $1001$ \\ 
\hline 
\hline 
$5$ & $1$& $2$&$2$ , $1$ ,  & $6$ \\ 
\hline 
$5$ & $2$& $3$&$4$, $3$, $1$  & $21$ \\ 
\hline 
$5$ & $3$& $4$&$7$, $7$, $3$, $1$  & $56$ \\ 
\hline 
$5$ & $4$& $5$&$12$, $14$, $7$, $3$, $1$  & $126$ \\ 
\hline 
$5$ & $5$& $5$& $19$, $25$, $14$, $8$, $3$ & $252$ \\ 
\hline 
$5$ & $6$& $6$&$29$, $42$, $26$, $16$, $7$, $1$ & $462$ \\ 
\hline 
$5$ & $7$& $6$&$42$ , $67$, $44$, $30$, $14$, $3$  & $792$ \\ 
\hline 
$5$ & $8$ & $6$ & $60$, $102$, $71$, $51$, $26$, $7$  & $1287$ \\ 
\hline 
$5$ & $9$ & $6$ & $83$, $150$, $109$, $83$, $44$, $14$  & $2002$ \\ 
\hline 
$5$ & $10$ & $7$ & $113$, $214$, $162$, $128$, $71$, $25$, $1$ & $3003$ \\ 
\hline 
\end{tabular}
}
\caption{Comparison for size of SDP \eqref{sdpsymred} and \eqref{sdptrigIdent} for different values of $n$ and $r$.}\label{comparisonSymRed}
\end{table}

\subsection*{Further structure exploitation}
We would like to point out that besides symmetry reduction another useful tool for solving large (semidefinite) optimization problems can be sparsity exploitation. In the context of polynomial optimization this topic has been studied before and results are available (cf. \cite{Lasserre2006, Wang2021}). Since our symmetry reduction works well for large $n$ but not necessarily for large $r$ it would be interesting to study sparsity patterns in our context. Our problem after symmetry reduction as given in \eqref{sdpsymred} is not sparse. However, one could enforce certain sparsity patterns on the data matrices while ensuring the problem stays feasible. This could lead to a more tractable optimization problem, whose solution may be less accurate but could still lead to decent approximation results. All theoretical guarantees on the speed of convergence will be lost in this case, though.

\section{Comparison to products of univariate minimum resolution kernels}
	\label{sec:comparison to univariate Jackson}
	In the following we will have a look at what kernels we get when we take the shortcut and multiply univariate kernels instead of solving the corresponding SDP. The clear advantage is that some optimal univariate kernels are available in closed form. Generating kernels as products means solving the corresponding SDP is unnecessary. Recall that in the univariate case kernels of the form 
	
	\begin{equation}\label{KKPM}
	K_r^{\mathrm{KPM}}(x,y) = 1+ 2\sum_{k = 1}^r g_{(k,r)}^{\mathrm{KPM}} T_k(x)T_k(y),
	\end{equation}
	for $g_{k,r}^{\mathrm{KPM}}$ as in \eqref{jackson}, have minimum resolution $\sigma_r$. The product of $n$ univariate degree $r$ kernels of the form \eqref{KKPM} results in an $n$-variate kernel of degree $nr$ that is feasible for the optimization problem \eqref{sdpsymred}. A natural question is to ask how these kernels compare to the ones obtained by solving the SDP. Consider the product of $n$ degree $r$ kernels
	
	%\begin{multline*}
	 \[
		\prod_{i = 1}^n \left(1+ 2\sum_{k = 1}^r g_{k,r}^{\mathrm{KPM}} T_k(x_i)T_k(y_i)\right) = 
		\sum_{\stackrel{\alpha \in \mathbb{N}^n_{nr}}{\alpha_i \le r, i \in [n]}} 2^{H(\alpha)}\tilde{g}_{\alpha}^{\mathrm{KPM}} T_{\alpha}(\textbf{x})T_{\alpha}(\textbf{y}),
	\]
	%\end{multline*}
	where
	\[
	\tilde{g}_{\alpha}^{\mathrm{KPM}} = \prod_{i = 1}^n g_{\alpha_i,r}^{\mathrm{KPM}}
	\]
	
	We know the resolution is 
	\[
	\sigma^2_r = \sum_{i=1}^n (1-g_{e_i}) = n(1-g_{e_1}). \]
	Thus, we can generate a feasible $n$-variate kernel with a degree $nr$ multiplying $n$ univariate degree $r$ kernels with minimum resolution and the corresponding resolution is 
	\[
	\sigma_{nr,\mathrm{KPM}}^2 = n\left(1-g_{1,r}^{\mathrm{KPM}}\right) \approx \frac{n\pi^2}{2(r+2)^2}.
	\]
	We would expect these to have a worse resolution than the kernels we obtain via solving the SDP where $\sigma_r^2$ is minimized. The reason for this is that the product kernels would be feasible to the same SDP with the additional set of constraints
	\[
	g_\alpha = 0 \text{ for all } \alpha \in \mathbb{N}^n_{nr} \text{ with } \alpha_i > r \text{ for some } i \in [n].
	\]
	In particular, we have the following result.
	
	\begin{proposition}\label{propConvRate}
		Fix $n \in \mathbb{N}$. For $ r \ge n$ we have
	\[
	\sigma^2_r \le n\left( 1- \cos \left( \frac{n\pi}{r+n}\right) \right) \sim \frac{n^3\pi^2}{2(r+n)^2} \mbox{ if $r \gg 0$}.
	\]
	
	\end{proposition}
	\begin{proof}
		
		Clearly, $\sigma_r^2 \le \sigma_{r-1}^2$ for any $r \ge 1$. Let now $k \in \mathbb{N}$ be such that $kn \le r \le (k+1)n$. Then we find
		\begin{align*}
			\sigma_r^2 \le \sigma_{nk}^2 \le \sigma_{nk,\mathrm{KPM}}^2 = n\left( 1 - \cos \left( \frac{\pi}{k+2} \right)  \right) \le n\left( 1 - \cos \left( \frac{\pi}{\frac{r}{n}+1}\right) \right) \sim \frac{n^3\pi^2}{2(r+n)^2},
		\end{align*}
		for $r \gg 0$.
	\end{proof}

	Looking at \cref{tab:ResultsComparison} we find that the values
	$\sigma_{r,\mathrm{KPM}}$ are larger than $\sigma_r$. 
	
	\begin{table}[ht]
		\centering
		\scalebox{0.98}{
		  \begin{tabular}{ |c||c|c|c|c|c|c| }
		  \hline
		   & \multicolumn{2}{|c|}{$n = 2$} & \multicolumn{2}{|c|}{$n = 3$} & \multicolumn{2}{|c|}{$n = 4$}\\
		  \hline
		  $r$& $\sigma^2_r $ & $\sigma^2_{r,\mathrm{KPM}}$ & $\sigma^2_r $ & $\sigma^2_{r,\mathrm{KPM}}$ & $\sigma^2_r $ & $\sigma^2_{r,\mathrm{KPM}}$ \\
		  \hline
		  \hline
		  1 & 1.5  & - & 2.5  & - & 3.5 & - \\
		  \hline
		   2 & 1  & 1 & 2 & - & 2.9310  & -\\
		   \hline
		   3 & 0.7378 & -  & 1.5 & 1.5 &  2.4561 & - \\
		   \hline
		   4 & 0.5487  & 0.5858 & 1.1823  & - &  1.9948 & 2 \\
		   \hline
		  5 & 0.4260  & - & 0.9451  & - & 1.6354  & -  \\
		  \hline
		   6 & 0.3395 & 0.3820 & 0.7764 & 0.8787 & 1.3605 & -  \\
		   \hline
		  7 & 0.2774  & - & 0.6474  & - & 1.1518  & -   \\
		  \hline
		  8 & 0.2299  & 0.2679 & 0.5461  & - & 0.9901& 1.1716\\
		  \hline
		  9 & 0.1939  & - & 0.4692  & 0.5729 & 0.8584 & - \\
		  \hline
		  10 & 0.1655  & 0.1981 & 0.4062  & - & 0.7524 & - \\
		  \hline
		  11 & 0.1431  & - & 0.3556  & - & 0.6648  & - \\
		  \hline
		  12 & 0.1248 & 0.1522 & 0.3136 & 0.4019 & 0.5917 & 0.7639  \\
		  \hline
		  13 & 0.1099 & - & 0.2787  & - & 0.5299  & -  \\
		  \hline
		  14 & 0.0975  & 0.1206 & 0.2493  & - &-  & -  \\
		  \hline
		  15 & 0.0871& - & 0.2243  & 0.2971 & -  & - \\
		  \hline
		  16 & 0.0782 & 0.0979 & 0.2028 & - & -  & 0.5359 \\
		  \hline
		  17 & 0.0706  & - & 0.1843 & - & -  & -\\
		  \hline
		  18 & 0.0641  & 0.0810 &0.1682 & 0.2284 & -  & - \\
		  \hline
		  19 & 0.0585  & - & 0.1541 & - & -  & -\\
		  \hline
		  20 & 0.0535  & 0.0681 & 0.1417  & - &   - & 0.3961  \\
		   \hline
		   21 & 0.0492 & - & 0.1307 & 0.1809 & - & -  \\
		  \hline
		   22 & 0.0453 & 0.0581 & 0.1209  & - & -  & - \\
		   \hline
		  23 & 0.0419  & - & - &  - &  -& -  \\
		  \hline
		  24 & 0.0389  & 0.0501 & - & 0.1468 &- & 0.3045  \\
		  \hline
		  25 & 0.0362  & - & - &  - & -  & - \\
		   \hline
		  30 & 0.0261  & 0.0341 & -& 0.1022&- & - \\
		  \hline
		  35 & 0.0197  & - &- & - & -&-\\
		  \hline
		  40 & 0.0154  & 0.0204 & - & -&- &0.1363 \\
		  \hline
		  45 & 0.0123 &- &- & 0.0511 & -&- \\
		  \hline
		  50 & 0.0101 & 0.01352 & -&- &- &- \\
		  \hline
		\end{tabular}
		}
		\caption{Computational results for $\sigma_r^2$ and $\sigma_{r,\mathrm{KPM}}^2$ for different values of $n$ and $r$ obtained by solving \eqref{sdpsymred}. }\label{tab:ResultsComparison}
	  \end{table} 
	
	Therefore, our generalization of the minimum resolution kernels to the multivariate case leads to better approximations than simply multiplying univariate kernels together. Also, for large $n$, i.e., the case for which our symmetry reduction is efficient, multiplying identical univariate kernels together is not always a feasible approach as the degree is always a multiple of $n$. In \cref{compProductJackson1} and \cref{compProductJackson2} we compare the errors of the approximation via the products of Jackson kernels with the approximation via minimum resolution for two different functions.

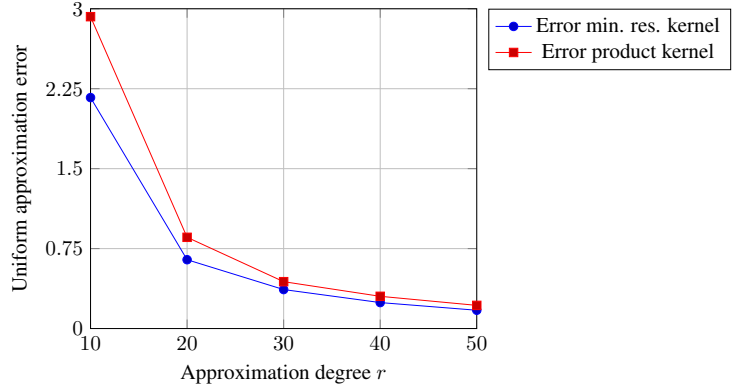
\begin{figure}
	%\begin{subfigure}[b]{0.5\linewidth}
	\centering
			\begin{tikzpicture}[scale=0.8]
			\begin{axis}[
			xlabel={Approximation degree $r$},
			ylabel={Uniform approximation error},
			ymajorgrids,
			xmajorgrids,
			xmin=10, xmax=50,
			ymin=0, ymax=3.0,
			xtick={10,20,30,40,50},
			ytick={0,0.75,1.5,2.25,3.0},
			legend style ={ at={(1.03,1)}, 
			anchor=north west, draw=black, 
			fill=white,align=left},
			]
			\addplot 
			coordinates {
				(10,2.1668)(20,0.6458)(30,0.3658)(40,0.2437)(50,0.1717)	};
			\addlegendentry{Error min. res. kernel}; 
			\addplot 
			coordinates {
				(10,2.9258)(20,0.8559)(30,0.4397)(40,0.3024)(50,0.2174)	};
			\addlegendentry{Error product kernel}; 
			\end{axis}
			\end{tikzpicture}
			\caption{Comparison of uniform approximation errors of several approximations of the function $q(\textbf{x}):= x_2 \sin (2\pi x_1)$. We plotted the errors for the kernel with minimal resolution $\sigma_{r}$ and for the product of two univariate degree $r/2$ kernels, i.e. $K_{r/2}^{\mathrm{KPM}}$ as in \eqref{KKPM}.}\label{compProductJackson1}
		\end{figure}

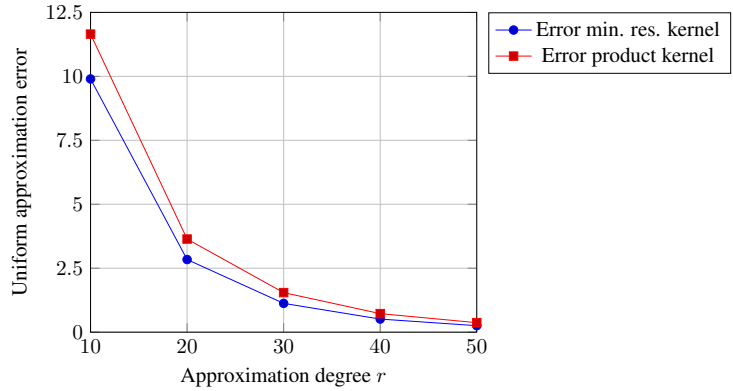
\begin{figure}
	%\begin{subfigure}[b]{0.5\linewidth}
	\centering
			\begin{tikzpicture}[scale=0.8]
				\begin{axis}[
				xlabel={Approximation degree $r$},
				ylabel={Uniform approximation error},
				ymajorgrids,
				xmajorgrids,
				xmin=10, xmax=50,
				ymin=0, ymax=12.5,
				xtick={10,20,30,40,50},
				ytick={0,2.5,5.0,7.5,10,12.5},
				legend style ={ at={(1.03,1)}, 
				anchor=north west, draw=black, 
				fill=white,align=left},
				]
				\addplot 
				coordinates {
					(10,9.8984)(20,2.8395)(30,1.1258)(40,0.5124)(50,0.2535)};
				\addlegendentry{Error min. res. kernel}; 
				\addplot 
				coordinates {
					(10,11.6499)(20,3.6402)(30,1.5460)(40,0.7250)(50,0.3717)};
				\addlegendentry{Error product kernel}; 
				\end{axis}
			\end{tikzpicture}
			\caption{Comparison of uniform approximation errors of several approximations of the peaks function $p(\textbf{x}):= 3(1-x_1)^2\exp(-x_1^2-(x_2+1)^2)-10(x_1/5-x_1^3-x_2^5)\exp(-x_1^2-x_2^2)-(1/3)\exp(-(x_1+1)^2-x_2^2)$. We plotted the errors for the kernel with minimal resolution $\sigma_{r}$ and for the product of two univariate degree $r/2$ kernels, i.e. $K_{r/2}^{\mathrm{KPM}}$ as in \eqref{KKPM}.}\label{compProductJackson2}
\end{figure}

  \section{Numerical computations}\label{sec:numComp}
	
	In this section we discuss the numerical computations that were conducted. 
	All code was written in the Julia programming language and is available on GitHub\footnote{see https://github.com/FelixKirschner/Approximation-Kernels}. At the same website we also list the coefficients of the kernel for various values of $n$ and $r$.
	We present some values of $\sigma_r^2$ for different values of $n$ and $r$ in \cref{tab:ResultsComparison}. 
	We also compare them to the resolution of the product of identical univariate minimum resolution kernels with the same degree.  
	The results show our method is superior to simple multiplication of identical univariate minimum resolution kernels. In \cref{fig:compSig} we plotted some of the values of $\sigma_r$ for different values of $n$. 
	
	For $n = 2$ we were able to compute the coefficients for up to $r = 50$ in a reasonable amount of time ($375.5$ seconds for $\sigma_{50}^2$ on an Apple M1 Pro with $32$GB of RAM). 
	After the symmetry reduction, the corresponding program contains two semidefinite matrix variables of order $676$ and  $650$ and has $1277$ constraints. 
	We used the CSDP solver version $6.2.0$\footnote{available at https://github.com/coin-or/Csdp} (see \cite{Borchers1999}) to compute these values. Without the symmetry reduction the program would have one matrix variable of order ${{52}\choose{2}} = 1326$. 
	We computed the values of $\sigma_r$ for up to $r = 22$ for $n = 3$ and $r = 13$ for $n = 4$. 
	%As an example the computation for $\bar \sigma_{13}$ for $n = 4$ took $\approx 48\, 000$ seconds. 
	In the latter case, i.e. $n=4, r= 13$, without the symmetry reduction the program contains one matrix of size $2380$. Using the symmetry reduction we can reduce the size to five matrices of the orders $194, 370, 192, 218$ and $38$. For  values of $n>2$ the limiting factor was time. 
	\subsection{Decoupling the degrees}
	Taking a look at problem \eqref{trig kernel SDP} it is clear the value of $r$ on the right-hand-side could be increased to obtain a kernel with potentially smaller resolution. Consider the following problem for fixed $r$ and $r^\prime$ such that $r^\prime \ge r$. 
	\begin{equation}
		 \sigma^2_{r,r^\prime} = \min_{g_\alpha \,:\, \alpha \in \mathbb{N}^n_r} \;\sum_{i=1}^n \left( 1 - g_{e_i} \right)
		\end{equation}
		subject to
	  \begin{eqnarray*}
			%t_i &\ge& (1-2g_{e_i}+g_{2e_i}) \quad \forall i \in [n] \\
			%t_i &\ge& -(1-2g_{e_i}+g_{2e_i}) \quad \forall i \in [n] \\
			1+ \sum_{\alpha \in \mathbb{N}^n_r\setminus \{0\}} 2^{H(\alpha)}g_\alpha \prod_{ i \in [n]} \cos (\alpha_i \phi_i) &=& \left[ \exp (\imath \alpha^T\boldsymbol\phi)\right]_{\alpha \in \mathbb{N}^n_{r^\prime}}^*M\left[ \exp (\imath \alpha^T\boldsymbol\phi)\right]_{\alpha \in \mathbb{N}^n_{r^\prime}} \\
			M &\succeq & 0.
		\end{eqnarray*}
	For $n = 2$ (resp. $n = 3$) we show how the resolution evolves for $r = 3, \dots, 10$ (resp. $r = 2, \dots, 10$) and $r^\prime = r, r+1, \dots, 20$ in Figure \ref{fig:compSigrrprime} (resp. \cref{fig:compSigrrprimen3}). We note that the optimal values in the case $n = 2$ seem to stabilize for $r^\prime \ge r + \lfloor \frac{r-1}{2}\rfloor$, whereas such a stabilization pattern may not be observed for $n \ge 3$. We leave further investigation in this direction for future research.
	
	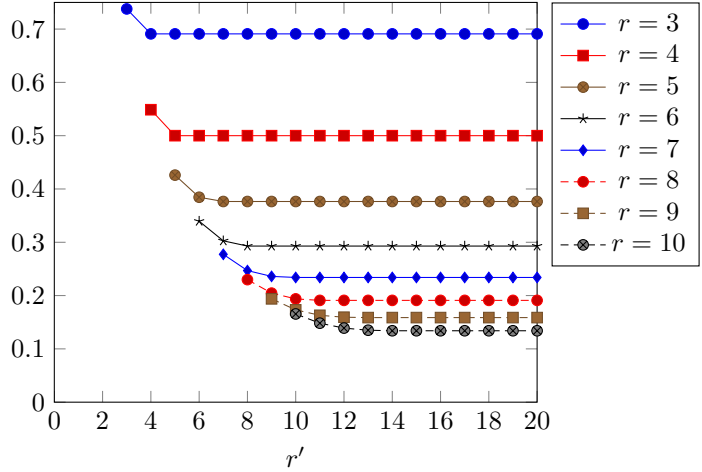
\begin{figure}
		%\begin{subfigure}[b]{0.5\linewidth}
		\centering
				\begin{tikzpicture}[scale=1.0]
				\begin{axis}[
				xlabel={$r^\prime$},
				%ylabel={$\sigma_r^2$},
				%ymajorgrids,
				%xmajorgrids,
				xmin=0, xmax=20,
				ymin=0, ymax=0.75,
				xtick={0,2,4,6,8,10,12,14,16,18,20},
				ytick={0,0.1,0.2,0.3,0.4,0.5,0.6,0.7},
				legend style ={ at={(1.03,1)}, 
        anchor=north west, draw=black, 
        fill=white,align=left},
				]
				%[
				%color=blue,
				%mark=square,
				%]
				\addplot 
				coordinates {
					(3,0.7378)(4,0.691)(5,0.691)(6,0.691)(7,0.691)(8,0.691)(9,0.691)(10,0.691)(11,0.691)(12,0.691)(13,0.691)(14,0.691)(15,0.691)(16,0.691)(17,0.691)(18,0.691)(19,0.691)(20,0.691)	};
					\addlegendentry{$r=3$};
				\addplot
				coordinates {
					(4,0.5487)(5,0.5)(6,0.5)(7,0.5)(8,0.5)(9,0.5)(10,0.5)(11,0.5)(12,0.5)(13,0.5)(14,0.5)(15,0.5)(16,0.5)(17,0.5)(18,0.5)(19,0.5)(20,0.5)
				};\addlegendentry{$r=4$};
				\addplot
				coordinates {
					(5,0.426)(6,0.3845)(7,0.3765)(8,0.3765)(9,0.3765)(10,0.3765)(11,0.3765)(12,0.3765)(13,0.3765)(14,0.3765)(15,0.3765)(16,0.3765)(17,0.3765)(18,0.3765)(19,0.3765)(20,0.3765)};\addlegendentry{$r=5$};
				\addplot
				coordinates {
					(6,0.3395)(7,0.3027)(8,0.2929)(9,0.2929)(10,0.2929)(11,0.2929)(12,0.2929)(13,0.2929)(14,0.2929)(15,0.2929)(16,0.2929)(17,0.2929)(18,0.2929)(19,0.2929)(20,0.2929)
				};\addlegendentry{$r=6$};
				\addplot
				coordinates {
					(7,0.2774)(8,0.2469)(9,0.236)(10,0.234)(11,0.234)(12,0.234)(13,0.234)(14,0.234)(15,0.234)(16,0.234)(17,0.234)(18,0.234)(19,0.234)(20,0.234)
				};\addlegendentry{$r=7$};
				\addplot
				coordinates {
					(8,0.2299)(9,0.2046)(10,0.1938)(11,0.191)(12,0.191)(13,0.191)(14,0.191)(15,0.191)(16,0.191)(17,0.191)(18,0.191)(19,0.191)(20,0.191)
				};\addlegendentry{$r=8$};
				\addplot
				coordinates {
					(9,0.1939)(10,0.1732)(11,0.163)(12,0.1594)(13,0.1587)(14,0.1587)(15,0.1587)(16,0.1587)(17,0.1587)(18,0.1587)(19,0.1587)(20,0.1587)
				};\addlegendentry{$r=9$};
				\addplot
				coordinates {
					(10,0.1655)(11,0.1482)(12,0.139)(13,0.135)(14,0.134)(15,0.134)(16,0.134)(17,0.134)(18,0.134)(19,0.134)(20,0.134)
				};\addlegendentry{$r=10$};
				\end{axis}
				\end{tikzpicture}
				\caption{Plot $\sigma_{r,r^\prime}^2$ vs $r^\prime$ for $r=3,\dots, 10$ and $r^\prime = r, r+1, \dots, 20$ and $n = 2$}\label{fig:compSigrrprime}
			\end{figure}
		
	%  TIKZ PICTURE?????
	%
	%
	%
	
%
%
%

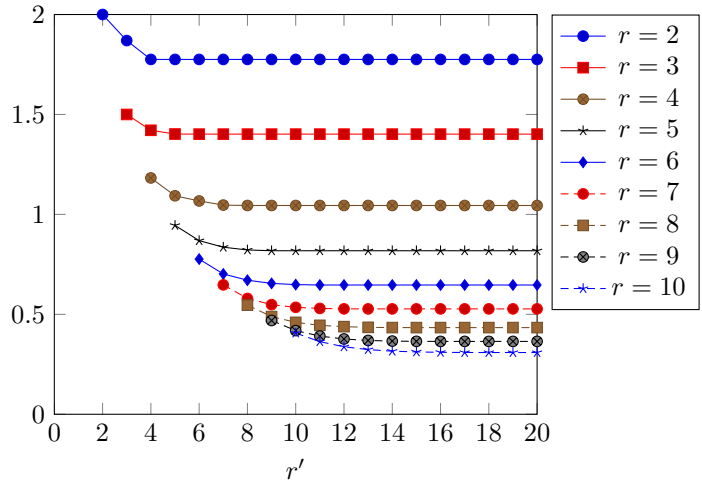
\begin{figure}
	%\begin{subfigure}[b]{0.5\linewidth}
	\centering
			\begin{tikzpicture}[scale=1.0]
			\begin{axis}[
			xlabel={$r^\prime$},
			xmin=0, xmax=20,
			ymin=0, ymax=2,
			xtick={0,2,4,6,8,10,12,14,16,18,20},
			ytick={0,0.5,1.0,1.5,2.0},
			legend style ={ at={(1.03,1)}, 
	anchor=north west, draw=black, 
	fill=white,align=left},
			]
			
			\addplot
			coordinates {
				(2,2.0)(3,1.8697)(4,1.7753)(5,1.7753)(6,1.7753)(7,1.7753)(8,1.7753)(9,1.7753)(10,1.7753)(11,1.7753)(12,1.7753)(13,1.7753)(14,1.7753)(15,1.7753)(16,1.7753)(17,1.7753)(18,1.7753)(19,1.7753)(20,1.7753)
			};\addlegendentry{$r=2$};
			\addplot
			coordinates {
				(3,1.5)(4,1.4208)(5,1.4021)(6,1.4016)(7,1.4016)(8,1.4016)(9,1.4016)(10,1.4016)(11,1.4016)(12,1.4016)(13,1.4016)(14,1.4016)(15,1.4016)(16,1.4016)(17,1.4016)(18,1.4016)(19,1.4016)(20,1.4016)};\addlegendentry{$r=3$};
			\addplot
			coordinates {
				(4,1.1823)(5,1.0931)(6,1.0672)(7,1.0466)(8,1.0445)(9,1.0445)(10,1.0445)(11,1.0445)(12,1.0445)(13,1.0445)(14,1.0445)(15,1.0445)(16,1.0445)(17,1.0445)(18,1.0445)(19,1.0445)(20,1.0445)
			};\addlegendentry{$r=4$};
			\addplot
			coordinates {
				(5,0.9451)(6,0.8687)(7,0.8359)(8,0.822)(9,0.8182)(10,0.8181)(11,0.8181)(12,0.8181)(13,0.8181)(14,0.8181)(15,0.8181)(16,0.8181)(17,0.8181)(18,0.8181)(19,0.8181)(20,0.8181)
			};\addlegendentry{$r=5$};
			\addplot
			coordinates {
				(6,0.7764)(7,0.702)(8,0.6707)(9,0.6548)(10,0.6487)(11,0.6467)(12,0.6467)(13,0.6467)(14,0.6467)(15,0.6467)(16 ,0.6467)(17,0.6467)(18,0.6467)(19,0.6467)(20,0.6467)
			};\addlegendentry{$r=6$};
			\addplot
			coordinates {
				(7,0.6474)(8,0.5794)(9,0.5485)(10,0.535)(11,0.5295)(12,0.5274)(13,0.5272)(14,0.5272)(15,0.5272)(16,0.5272)(17,0.5272)(18,0.5272)(19,0.5272)(20,0.5272)
			};\addlegendentry{$r=7$};
			\addplot
			coordinates {
				(8,0.5461)(9,0.4888)(10,0.4598)(11,0.4449)(12,0.4384)(13,0.4349)(14,0.4337)(15,0.4336)(16,0.4336)(17,0.4336)(18,0.4336)(19,0.4336)(20,0.4336)
			};\addlegendentry{$r=8$};
			\addplot
			coordinates {
				(9,0.4692)(10,0.4186)(11,0.3911)(12,0.3766)(13,0.3692)(14,0.3661)(15,0.3648)(16,0.3645)(17,0.3645)(18,0.3645)(19,0.3645)(20,0.3645)
			};\addlegendentry{$r=9$};
			\addplot
			coordinates {
				(10,0.4062)(11,0.3634)(12,0.3377)(13,0.3235)(14,0.3157)(15,0.3119)(16,0.3099)(17,0.3091)(18,0.3089)(19,0.3089)(20,0.3089)
			};\addlegendentry{$r=10$};
			\end{axis}
			\end{tikzpicture}
			\caption{Plot $\sigma_{r,r^\prime}^2$ vs $r^\prime$ for $r=2,\dots, 10$ and $r^\prime = r, r+1, \dots, 20$ and $n = 3$}\label{fig:compSigrrprimen3}
		\end{figure}

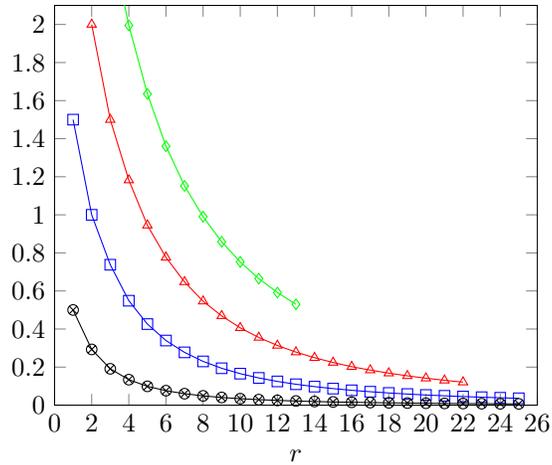
\begin{figure}
	%\begin{subfigure}[b]{0.5\linewidth}
	\centering
			\begin{tikzpicture}[scale=1.0]
			\begin{axis}[
			xlabel={$r$},
			xmin=0, xmax=26,
			ymin=0, ymax=2.1,
			xtick={0,2,4,6,8,10,12,14,16,18,20,22,24,26},
			ytick={0,0.2,0.4,0.6,0.8,1,1.2,1.4,1.6,1.8,2.0},
			]
			\addplot[
			color=blue,
			mark=square,
			]
			coordinates {
				(1,1.5)(2,1)(3,0.737831)(4,0.548709)(5,0.426016)(6,0.339528)(7,0.277385)(8,0.229869)(9,0.193916)(10,0.165505)(11,0.143117)(12,0.124814)(13,0.109943)(14,0.097454)(15, 0.087057)(16,0.078179)(17,0.070646)(18,0.064099)(19,0.058459)(20,0.053501)(21,0.049175)
			(22,0.045327)(23,0.041934)(24,0.038891)(25,0.036182)		};
			\addplot[
			color = red,
			mark = triangle,
			]
			coordinates {
				(2,2)(3,1.5)(4,1.182282)(5,0.945126)(6,0.776365)(7,0.647362)(8,0.546111)(9,0.469164)(10,0.406163)(11,0.355589)(12,0.313636)(13,0.278721)(14,0.249285)(15,0.2242)(16,0.2022)(17,0.1843)(18,0.1682)(19,0.1541)(20,0.1417)(21,0.1307)(22,0.1209)
			};
			\addplot[
			color = green,
			mark = diamond,
			]
			coordinates {
				(2,2.9309)(3,2.4561)(4,1.9948)(5,1.6353)(6,1.3605)(7,1.1518)(8,0.99005)(9,0.8584)(10,0.7524)(11,0.6648)(12,0.5917)(13,0.5299)};
			\addplot[
				color = black,
				mark = otimes,
			]
			coordinates {
				(1,0.5)(2,0.2929)(3,0.191)(4,0.134)(5,0.099)(6,0.0761)(7,0.0603)(8,0.0489)(9,0.0405)(10,0.0341)(11,0.0291)(12,0.0251)(13,0.0219)(14,0.0192)(15,0.017)(16,0.0152)(17,0.0136)(18,0.0123)(19,0.0112)(20,0.0102)(21,0.0093)(22,0.0086)(23,0.0079)(24,0.0073)(25,0.0068)
			};
			\end{axis}
			\end{tikzpicture}
			\caption{Plot $\sigma_r^2$ vs $r$ for $n=1,2,3,4$ ($\otimes ,\square, \triangle, \diamond$ resp.) }
			\label{fig:compSig}
		\end{figure}

\section{Concluding remarks}
We have shown how to construct polynomial approximation kernels with minimal resolution on the hypercube.
A major open question is if one may find closed form solutions of the semidefinite programs that yield these kernels.

These type of results are also of independent interest in the study of SDP hierarchies for polynomial optimization on the hypercube, as shown recently by Laurent and Slot \cite{Laurent2021}. In particular, our kernels may be useful to study hierarchies of the Lasserre-type \cite{Lasserre2001} on the hypercube (see also \cite{Klerk2017, Klerk2010}).

The advantage of our approach over the multiplication of univariate minimum resolution kernels is that it is more efficient (fewer coefficients needed for the same quality approximation), while the clear disadvantage is that we have no closed form solution for the coefficients. Having said that, the tables of coefficients only have to computed once using SDP, and we provide a partial list online\footnote{Available at https://github.com/FelixKirschner/Approximation-Kernels/tree/master/SigmaKernels}. Moreover, our approach should become more viable in practice as SDP solvers continue to improve, allowing to compute the coefficients of the kernels in higher dimensions and for larger values of $r$.
\\	
\section*{Acknowledgments}
The authors would like to thank Yuan Xu and Simon Foucart for helpful comments on the paper. 
	Also, we want to thank Monique Laurent and Lucas Slot for fruitful discussions on different angles on the subject. 
	Finally, the first author would like to thank Daniel Brosch for helpful discussions about symmetry reduction.

\bibliographystyle{plain}

\end{document}